\documentclass[11pt]{amsart}
\usepackage{etex}
\usepackage{pst-all}
\usepackage{amsmath}
\usepackage{amsfonts}
\usepackage{amssymb}
\usepackage{multicol}
\usepackage{verbatim}
\usepackage{amsthm}
\usepackage{amscd}
\usepackage{pb-diagram}
\usepackage[mathscr]{eucal}

\usepackage{amsmath}

\usepackage{amsfonts}

\usepackage{amssymb}

\usepackage{amsthm}

\usepackage{graphics}
\usepackage{tikz}
\usetikzlibrary{automata, arrows}
\usetikzlibrary{arrows,positioning,shapes.geometric}
\usetikzlibrary{decorations.text}
\usepackage{graphics}

\usepackage{setspace}

\newcommand{\ZZ}{\mathbb{Z}}

\newcommand{\CC}{\mathbb{C}}

\newcommand{\NN}{\mathbb{N}}

\newcommand{\wt}{\mathrm{wt}}

\newcommand{\cl}{\mathrm{cl}}

\newcommand{\ffbox}[1]{
\setbox9=\hbox{$\scriptstyle\overline{1}$}
\framebox[20pt][c]{\rule{0mm}{\ht9}${\scriptstyle #1}$}
}

\newcommand{\QQ}{\mathbb{Q}}

\newcommand{\Glie}{\mathfrak{g}}

\newcommand{\Hlie}{\mathfrak{h}}

\newcommand{\U}{\mathcal{U}}

\newtheorem{thm}{Theorem}[section]

\newtheorem{defi}[thm]{Definition}

\newtheorem{prop}[thm]{Proposition}

\newtheorem{lem}[thm]{Lemma}

\newtheorem{rem}[thm]{Remark}

\newtheorem{ex}[thm]{Example}

\usepackage[all]{xy}
\title[Extremal loop weight modules and tensor products]{Extremal loop weight modules and tensor products for quantum toroidal algebras}

\author[Mathieu Mansuy]{Mathieu Mansuy}
\address{Univ. Paris-Diderot-Paris 7, IMJ - PRG CNRS UMR 7586, B\^at. Sophie Germain, Case 7012, 75205 Paris Cedex 13, FRANCE}
\email{mansuy@math.jussieu.fr}

\begin{document}

\begin{abstract}
We define integrable representations of quantum toroidal algebras of type $A$ by tensor product, using the Drinfeld ``coproduct''. This allow us to recover the vector representations recently introduced by Feigin-Jimbo-Miwa-Mukhin \cite{feigin_representations_2013} and constructed by the author \cite{mansuy_quantum_2012} as a subfamily of extremal loop weight modules. In addition we get new extremal loop weight modules as subquotients of tensor powers of vector representations. As an application we obtain finite-dimensional representations of quantum toroidal algebras by specializing the quantum parameter at roots of unity.
\end{abstract}

\maketitle

\tableofcontents

\section{Introduction}

Consider a finite-dimensional simple Lie algebra $\Glie$. The quantum toroidal algebra (or double affine quantum algebra) $\U_q(\Glie^{tor})$ associated to $\Glie$ was first introduced by Ginzburg-Kapranov-Vasserot in type A \cite{ginzburg_langlands_1995} and then in the general context in \cite{jing_quantum_1998, nakajima_quiver_2001}. It is built from a copy of the quantum affine algebra $\U_q(\hat{\Glie})$ and a copy of the quantum loop algebra (quantum affine algebra without derivation element), by the Drinfeld quantum affinization process. Quantum toroidal algebras are the subject of intensive research with applications in mathematics and physic (see for example \cite{ hernandez_representations_2005, hernandez_quantum_2009, miki_representations_2000, varagnolo_double-loop_1998} and more recently \cite{feigin_representations_2013, feigin_branching_2013, nekrasov_quantum_2013}).

The representation theory of quantum toroidal algebras is very rich and promising. There exists a notion of loop highest weight modules (see \cite{miki_representations_2000, nakajima_quiver_2001}), analogs of the highest weight modules in the toroidal case. It is also possible to introduce the counterpart of extremal weight $\U_q(\hat{\Glie})$-modules for quantum toroidal algebras. Extremal weight modules are integrable representations of $\U_q(\hat{\Glie})$ introduced by Kashiwara \cite{kashiwara_crystal_1994}. They are the subject of numerous papers (see \cite{beck_crystal_2004, hernandez_level_2006, kashiwara_level-zero_2002} and references therein) and have a particular importance because some of them have finite-dimensional quotients. In the spirit of works of Kashiwara, Hernandez \cite{hernandez_quantum_2009} introduced the notion of extremal loop weight modules for $\U_q(\Glie^{tor})$: they are integrable representations generated by an extremal vector (in the sense of Kashiwara \cite{kashiwara_crystal_1994}) for the copy of the quantum affine algebra $\U_q(\hat{\Glie})$ in $\U_q(\Glie^{tor})$. The main motivation is to construct finite-dimensional representations of quantum toroidal algebras by specializing the quantum parameter at roots of unity.

In a recent work \cite{mansuy_quantum_2012} the author constructed the first known family of extremal loop weight modules for quantum toroidal algebras of type $A$: when $n=2r+1$ is odd (with $r \geq 1$), we have defined extremal (fundamental) loop weight representations associated to the nodes $\ell = 1, r+1$ and $n$ of the Dynkin diagram. This construction is based on monomial realizations of crystal bases of extremal fundamental weight $\U_q(\hat{\Glie})$-modules.

\medskip

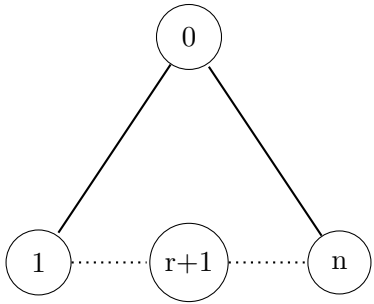
\begin{center}
\begin{tikzpicture}[>=stealth',shorten >=1pt,auto,node distance=2.8cm]
\tikzstyle{point}=[circle,draw]
\tikzstyle{ligne}=[thick]
\tikzstyle{pointille}=[thick,dotted]

\node (2) at ( 0,3) [point] {\ 0 \ };
\node (3) at ( -2, 0) [point] {\ 1 \ };
\node (4) at ( 0, 0) [point] {r+1};
\node (5) at ( 2, 0) [point] {\ n \ };

\draw [ligne] (2) -- (3);
\draw [pointille] (3) -- (4);
\draw [pointille] (4) -- (5);
\draw [ligne] (5) -- (2);
\end{tikzpicture}
\end{center}

\medskip

The construction done in the article is inspired by the original study of extremal weight modules: these representations are closely related to a tensor product of a simple highest weight module and a simple lowest weight module (see \cite{kashiwara_crystal_1994}). In the present paper we consider the tensor product of a simple loop highest weight module and a simple loop lowest weight module, depending of a node $\ell \neq 0$ of the Dynkin diagram and non zero complex parameters $a, b$. No Hopf algebra structure is known for $\U_q(sl_{n+1}^{tor})$ but there exists a coproduct $\Delta_D$ (the Drinfeld coproduct) which involve infinite sums. Under some conditions on the parameters $a$ and $b$ and using analogous features as in \cite{hernandez_drinfeld_2007}, we show that $\Delta_D$ is well-defined on the tensor product and endow it to a structure of integrable $\U_q(sl_{n+1}^{tor})$-module (Theorem \ref{tpexist}). The representation hence obtained is not extremal in general (Proposition \ref{propnecon}): a necessary condition is that
\begin{align*}
\begin{cases}
n \text{ is even and } \ell = 1, n \\
\text{or}\\
n = 2r+1 \text{ is odd and } \ell = 1, r+1, n
\end{cases}
\end{align*}
and $b=aq^{\min(\ell, n+1-\ell)}$. When $\ell=1$ or $n$, we show that this representation (denoted $V_{a}$) is loop extremal (Theorem \ref{repvect2}): we recover here the \textit{vector representation} defined in \cite{feigin_representations_2013} by Feigin-Jimbo-Miwa-Mukhin for the $d$-deformation $\U_{q,d}(sl_{n+1}^{tor})$ of the quantum toroidal algebra $\U_q(sl_{n+1}^{tor})$, and constructed by the author in \cite{mansuy_quantum_2012}.

We want to construct a large family of extremal loop weight $\U_q(sl_{n+1}^{tor})$-modules. A well-known result of representation theory says that every finite-dimensional simple module over the quantum group $\U_q(sl_{n+1})$ is a subquotient of tensor powers of vector representations (see for example \cite{hong_introduction_2002}). Motivated by this fact, we consider the analogous situation for $\U_q(sl_{n+1}^{tor})$: we study tensor powers of vector representations
$V_{a_1} \otimes \cdots \otimes V_{a_k} $ with $a_1, \cdots , a_k$ non zero complex parameters. Using analogous features as in \cite{feigin_representations_2013} we show that, under some conditions on the $a_i$'s, $\Delta_D$ endows this tensor product with a structure of $\U_q(sl_{n+1}^{tor})$-module. When the parameters $a_i$ form a $q$-segment, we recover all the extremal loop weight modules constructed in \cite{mansuy_quantum_2012} as subquotients of this tensor product (Theorem \ref{tpeflwm} and Proposition \ref{proprecrep}). We also obtain new extremal loop weight representations of $\U_q(sl_{n+1}^{tor})$ when the parameters are chosen generic (Theorem \ref{thmexgen} and Theorem \ref{thmexgen2}). By specialization, we get new finite-dimensional representations of quantum toroidal algebras at roots of unity.

Tensor products of vector representations are also studied in \cite{feigin_representations_2013} for $\U_{q,d}(sl_{n+1}^{tor})$. Let us point out that their approach and motivations are different than ours: they define an action of $\U_{q,d}(sl_{n+1}^{tor})$ on an infinite tensor product of vector representations by a semi-infinite wedge construction. They construct in this way a large family of irreducible loop lowest weight representations of $\U_{q,d}(sl_{n+1}^{tor})$. As an application they obtain combinatorial descriptions of representations of $\U_q(\hat{sl}_{n+1})$ by restriction. Other applications are expected in conformal field theory for the AGT conjecture. Our motivations are different: our aim is to find a process to construct a large family of extremal loop weight modules for all the quantum affinizations. This is the main reason why we work with $\U_q(sl_{n+1}^{tor})$ instead of its $d$-deformation. In fact $d$-deformations of quantum toroidal algebras not of type $A$ are not known. We show on an example that the features we develop in the paper can hold for quantum toroidal algebras of other types: we construct an extremal loop weight module for the quantum toroidal algebra of type $D_4$ (Theorem \ref{extypd}). This is the first example of such module in type different to $A$.

The paper is organized as follows.

In Section 2 we recall the main definitions of quantum affine algebras $\U_q(\hat{sl}_{n+1})$ and quantum toroidal algebras $\U_q(sl_{n+1}^{tor})$ and we briefly review their representation theory. In particular one defines the extremal weight modules and the extremal loop weight modules. In Section 3 we study tensor products of simple $\ell$-highest weight modules and simple $\ell$-lowest weight modules. Section 4 and Section 5 are devoted to the study of tensor products of vector representations. In Section 6 we construct an example of extremal loop weight module for the quantum toroidal algebra of type $D_4$.\\

\textbf{Acknowledgements:}
I would like to thank my advisor David Hernandez for his encouragements and for his numerous and precious comments. I want also to thank Alexandre Bouayad for his accurate and helpful remarks. I thank Vyjayanthi Chari for her interest about this work and her hospitality during my stay at the University of Riverside.

\section{Background}

\subsection{Cartan matrix} \nocite{kac_infinite-dimensional_1990}

Fix a natural number $n \geq 2$. Set $I=\{0,\dots,n\}$. It will be often identified with the set $\ZZ / (n+1) \ZZ$, and we will denote by $\bar{i}$ the class of an integer $i$ in $\ZZ / (n+1) \ZZ$. Let $C=(C_{i,j})_{i, j \in I}$ be the Cartan matrix of type $A_n^{(1)}$,
$$ C_{i,i} = 2 \text{ , } C_{i,i+1} = C_{i+1, i} = -1 \text{ and } C_{i,j} = 0 \text{ if } j \neq i, i \pm 1.$$

\noindent Set $I_0 = \{1, \dots, n\}$. In particular, $(C_{i,j})_{i,j \in I_0}$ is the Cartan matrix of finite type $A_n$. Consider the vector space of dimension $n+2$
 $$\Hlie = \QQ h_0 \oplus \QQ h_1 \oplus \dots \oplus \QQ h_n \oplus \QQ d$$
\noindent and the linear functions $\alpha_i$ (the simple roots), $\Lambda_i$ (the fundamental weights) on $\Hlie$ given by ($i,j \in I$)
\begin{eqnarray*}
\alpha_i (h_j)= C_{j,i}, & \alpha_i(d)= \delta_{0,i},\\
\Lambda_i(h_j) = \delta_{i,j}, & \Lambda_i(d)= 0.
\end{eqnarray*}

Denote by $\Pi=\{\alpha_0,\dots,\alpha_n\}\subset \Hlie^*$ the set of
simple roots and $\Pi^{\vee}=\{h_0,\dots, h_n\}\subset \Hlie$ the set of simple coroots. Let $P =\{\lambda \in\Hlie^* \mid \text{$\lambda(h_i)\in\ZZ$ for any $i\in I$}\}$ be the weight lattice and $P^+=\{\lambda \in P \mid \text{$\lambda(h_i)\geq 0$ for any $i\in I$}\}$ the semigroup of dominant weights. Let $Q=\oplus_{i\in I} \ZZ \alpha_i\subset P$ (the root lattice) and $Q^+=\sum_{i\in I}\NN \alpha_i\subset Q$. For $\lambda,\mu\in \Hlie^*$, write $\lambda \geq \mu$ if $\lambda-\mu\in Q^+$.

Denote by $W$ the affine Weyl group: this is the subgroup of $GL(\Hlie^*)$ generated by the simple reflections $s_i\in GL(\Hlie^*)$ defined by $s_i(\lambda)=\lambda-\lambda(h_i)\alpha_i$ for all $i\in I$.

Let $c= h_0 + \dots + h_n$ and $\delta = \alpha_0 + \dots + \alpha_n$. We have
$$\{\omega \in P \mid \omega(h_i)=0 \text{ for all } i \in I \}=\QQ \delta.$$
Put $P_{\cl}=P/ \QQ \delta$ and denote by $\cl : P \rightarrow P_{\cl}$ the canonical projection. Denote by $P^0=\{\lambda\in P\mid \lambda(c)=0\}$ the set of level $0$ weights.

\subsection{Quantum affine algebra $\U_q(\hat{sl}_{n+1})$}

\subsubsection{Definition}
In the whole article $t \in \CC$ is fixed such that $q = e^{t} \in \CC^*$ is not a root of unity. For $l\in\ZZ, r\geq 0, m\geq m'\geq 0$ we set
$$[l]_q=\frac{q^l-q^{-l}}{q-q^{-1}}\in\ZZ[q^{\pm 1}],\  [r]_q!=[r]_q[r-1]_q\dots[1]_q,\ \begin{bmatrix}m\\m'\end{bmatrix}_q=\frac{[m]_q!}{[m-m']_q![m']_q!}.$$

\begin{defi} The quantum affine algebra $\U_q(\hat{sl}_{n+1})$ is the $\CC$-algebra with generators $k_h$ $(h\in \Hlie)$, $x_i^{\pm}$ $(i\in I)$ and relations
\begin{equation*}k_hk_{h'}=k_{h+h'}\text{ , }k_0=1,\end{equation*}
\begin{equation*}k_hx_j^{\pm}k_{-h}=q^{\pm \alpha_j(h)}x_j^{\pm},\end{equation*}
\begin{equation*}[x_i^+,x_j^-]=\delta_{i,j}\frac{k_i-k_{i}^{-1}}{q-q^{-1}},\end{equation*}
\begin{equation*}
(x_i^{\pm})^{(2)}x_{i+1}^{\pm} - x_i^{\pm}x_{i+1}^{\pm}x_i^{\pm} + x_{i+1}^{\pm}(x_i^{\pm})^{(2)} = 0.\end{equation*}
\end{defi}

\noindent We use here the notation $k_i^{\pm 1} = k_{\pm h_i}$ and for all $r \geq 0$ we set $(x_i^\pm)^{(r)} = \frac{(x_i^\pm)^r}{[r]_q!}$. One defines a coproduct on $\U_q(\hat{sl}_{n+1})$ by setting
$$\Delta(k_h)=k_h\otimes k_h,$$
$$\Delta(x_i^+)=x_i^+\otimes 1 + k_i^+\otimes x_i^+\text{ , }\Delta(x_i^-)=x_i^-\otimes k_i^- + 1\otimes x_i^-.$$

\subsubsection{Subalgebras}
Let $\U_q(\hat{sl}_{n+1})'$ be the subalgebra of $\U_q(\hat{sl}_{n+1})$ generated by $x_{i}^{\pm}$ and $k_h$ with $h\in \sum \QQ h_{i}$. This has $P_\cl$ as a weight lattice.

Let $\U_q(\hat{sl}_{n+1})^+$ (resp. $\U_q(\hat{sl}_{n+1})^-$, $\U_q(\Hlie)$) be the subalgebra of $\U_q(\hat{sl}_{n+1})$ generated by the $x_i^+$ (resp. the $x_i^-$, resp the $k_h$). We have a triangular decomposition of $\U_q(\hat{sl}_{n+1})$. 

\begin{thm}\cite{lusztig_introduction_1993} We have an isomorphism of vector spaces
$$\U_q(\hat{sl}_{n+1}) \simeq \U_q(\hat{sl}_{n+1})^- \otimes \U_q(\Hlie) \otimes \U_q(\hat{sl}_{n+1})^+.$$
\end{thm}

\subsection{Representations of $\U_q(\hat{sl}_{n+1})$}

\subsubsection{Highest weight modules} For $V$ a representation of $\U_q(\hat{sl}_{n+1})$ and $\nu\in P$, the weight space $V_{\nu}$ of $V$ is
$$V_\nu = \{v\in V \mid k_h \cdot v = q^{\nu(h)}v \text{ for any } h\in \Hlie\}.$$
Set $\wt(V) = \{\nu \in P \mid V_\nu \neq \{0 \} \}$.

\begin{defi} A representation $V$ is said to be in the category $\mathcal{O}$ if 
\begin{enumerate}
\item[(i)] it admits a weight space decomposition $V = \bigoplus_{\nu\in P} V_\nu$,
\item[(ii)] $V_\nu$ is finite-dimensional for all $\nu$,
\item[(iii)] $\wt(V) \subset \bigcup_{j=1\cdots N}\{\nu \mid \nu\leq \lambda_j\}$ for some $\lambda_1,\cdots,\lambda_N\in P$.
\end{enumerate}
\end{defi}

For $\lambda\in P$, a representation $V$ is said to be of highest weight $\lambda$ if there is $v\in V_\lambda$ such that for any $i\in I x_i^+ \cdot v = 0$ and $\U_q(\hat{sl}_{n+1}) \cdot v = V$. Such a representation is in the category $\mathcal{O}$. Furthermore there is a unique simple highest weight module of highest weight $\lambda$.

\begin{defi}\label{defint} A representation  $V$ is said to be integrable if
\begin{enumerate}
\item[(i)] it admits a weight space decomposition $V = \bigoplus_{\nu\in P} V_\nu$,
\item[(ii)] for any $ v \in V$ and $i \in I$, there is $N \geq 1$ such that $(x_i^\pm)^{N} \cdot v = 0$.
\end{enumerate}
\end{defi}

\begin{thm}\cite{lusztig_introduction_1993}
The simple highest weight module of highest weight $\lambda$ is integrable if and only if $\lambda$ is dominant.
\end{thm}

\subsubsection{Extremal weight modules}

\begin{defi}\cite{kashiwara_crystal_1994} Let $V$ be an integrable $\U_q(\hat{sl}_{n+1})$-module $V$ and $\lambda\in P$.
\begin{itemize}
\item[(i)] A vector $v\in V_{\lambda}$ is called $i$-extremal with $i \in I$ if
$$x_i^{\pm} \cdot v = 0 \text{ if } \pm \lambda(h_i) \geq 0.$$
We set in this case $S_i(v) = (x_i^\mp)^{(\pm\lambda(h_i))} \cdot v$.
\item[(ii)] The vector $v\in V_{\lambda}$ is called extremal of weight $\lambda$ if for all $l \geq 0$, $S_{i_1} \circ \cdots \circ S_{i_l} (v)$ is $i$-extremal for any $i, i_1, \cdots, i_l \in I$.
\end{itemize}
\end{defi}

\begin{defi}\cite{kashiwara_crystal_1994} For $\lambda\in P$, the extremal weight module $V(\lambda)$ of extremal weight $\lambda$ is the $\U_q(\hat{sl}_{n+1})$-module generated by a vector $v_{\lambda}$ with the defining relations that $v_{\lambda}$ is extremal of weight $\lambda$.\end{defi}

\begin{thm}\label{thmkas}\cite{kashiwara_crystal_1994} The module $V(\lambda)$ is integrable and has a crystal basis $\mathcal{B}(\lambda)$.\end{thm}

\begin{ex} If $\lambda$ is dominant, $V(\lambda)$ is the simple highest weight module of highest weight $\lambda$. \end{ex}

\begin{rem}\label{remkash}
To prove Theorem \ref{thmkas}, the tensor product $V(\lambda_+) \otimes V(-\lambda_-)$ is considered in \cite[Section 8.2]{kashiwara_crystal_1994}, with
$$\lambda_+ = \sum_{\lambda(h_i) \geq 0} \lambda(h_i) \Lambda_i \text{ and } \lambda_- = \lambda_+ - \lambda \in P^+.$$
\end{rem}

For any $1 \leq \ell \leq n$ set $ \varpi_{\ell} = \Lambda_\ell - \Lambda_0 $.

\begin{thm}\cite{kashiwara_level-zero_2002}
\begin{enumerate}
\item[(i)] $V(\varpi_\ell)$ is an irreducible $\U_q(\hat{sl}_{n+1})$-module.
\item[(ii)] Any non-zero integrable $\U_q(\hat{sl}_{n+1})$-module generated by an extremal weight vector of weight $\varpi_\ell$ is isomorphic to $V(\varpi_\ell)$.
\item[(iii)] As a $\U_q(\hat{sl}_{n+1})'$-module, $V(\varpi_\ell)$ admits an irreducible and finite-dimensional quotient.
\end{enumerate}
\end{thm}

\subsection{Quantum toroidal algebra $\U_q(sl_{n+1}^{tor})$}

\subsubsection{Definition} We recall the definition of the quantum toroidal algebra $\U_q(sl_{n+1}^{tor})$ (without central charge) in terms of currents.

\begin{defi}\label{defqta}\cite{ginzburg_langlands_1995}
The quantum toroidal algebra $\U_q(sl_{n+1}^{tor})$ is the $\CC$-algebra with generators $x_{i,r}^\pm$  ($i \in I, r \in \ZZ$), $k_h$ ($h \in \Hlie$), $h_{i,m}$ ($i \in I, m \in \ZZ-\{0\}$) and the following defining relations ($i,j \in I, h, h' \in \Hlie$):
\begin{equation*}
k_h k_{h'} = k_{h+h'} \text{ , } k_0 = 1,
\end{equation*}
\begin{equation*}
\phi^\pm_i(z)\phi^\pm_j (w) = \phi^\pm_j(w)\phi^\pm_i (z) \text{ , } \phi^-_i(z)\phi^+_j (w) = \phi^+_j(w)\phi^-_i (z),
\end{equation*}
\begin{equation}\label{equa1}
(w -q^{\pm C_{ij}}z) \phi_i^+(z)x_j^\pm(w)= 
(q^{\pm C_{ij}}w - z) x_j^\pm(w)\phi_i^+(z),
\end{equation}
\begin{equation}\label{equa2}
(w -q^{\pm C_{ij}}z) \phi_i^-(z)x_j^\pm(w)= 
(q^{\pm C_{ij}}w - z) x_j^\pm(w)\phi_i^-(z),
\end{equation}
\begin{equation*}
[x_i^+(z),x_j^-(w)]=\frac{\delta_{i,j}}{q-q^{-1}}(\delta\bigl(\frac{w}{z}\bigr)\phi_i^+(w) -\delta\bigl(\frac{z}{w}\bigr)\phi_i^-(z)),
\end{equation*}
\begin{equation}\label{equa3}
(z-q^{\pm C_{ij}}w) x_i^\pm(z)x_j^\pm(w)=(q^{\pm C_{ij}}z-w)x_j^\pm(w)x_i^\pm(z),
\end{equation}
\begin{eqnarray*}
\begin{array}{c}
x_i^\pm(z_1)x_i^\pm(z_2)x_{j}^\pm(w)-(q+q^{-1})x_i^\pm(z_1)x_{j}^\pm(w) x_i^\pm(z_2) \\
+x_{j}^\pm(w) x_i^\pm(z_1) x_i^\pm(z_2) + (z_1\leftrightarrow  z_2)=0,
\end{array}
\end{eqnarray*}
for $j = i+1, i-1$, and $[x_i^\pm(z),x_j^\pm(w)]=0$ for $i \neq j,j+1,j-1$.

We use here the formal power series $\delta(z) = \sum_{s \in \ZZ} z^{s}$ and

\begin{eqnarray*}
x_i^\pm(z) = \sum_{r \in \ZZ} x_{i,r}^\pm z^{r},
\end{eqnarray*}
\begin{eqnarray*}
\phi_i^{\pm}(z) = \sum_{m \geq 0} \phi_{i,\pm m}^\pm z^{\pm m} = k_i^{\pm 1} \exp(\pm(q-q^{-1}) \sum_{m' \geq 1} h_{i, \pm m'} z^{\pm m'}).
\end{eqnarray*}
\end{defi}


\subsubsection{Automorphisms of $\U_q(sl_{n+1}^{tor})$}

For $b\in\CC^\ast$, let $t_b$ be the automorphism of $\U_q(sl_{n+1}^{tor})$ which sends the series 
$x_i^{\pm}(z)$ and $ \phi_i^{\pm}(z)$ to $x_i^{\pm}(bz)$ and $ \phi_i^{\pm}(bz)$ respectively, and $k_h$ to $k_h$.

Let us consider the automorphism of the Dynkin diagram of type $ A_{n}^{(1)} $ corresponding to the rotation, which sends $i$ to $ i + 1 $ for $i \in I$. It defines an automorphism $\theta_\Hlie$ of $\Hlie$ by sending $h_i, d$ to $h_{i+1}, d$. This also defines an algebra automorphism $\theta$ of $\U_q(sl_{n+1}^{tor})$ which sends the series $x_i^{\pm}(z)$ and $ \phi_i^{\pm}(z)$ to $x_{i+1}^{\pm}(z)$ and $ \phi_{i+1}^{\pm}(bz)$ respectively, and $ k_{h} $ to $k_{\theta_\Hlie(h)} $.

Let us consider the automorphism of the Dynkin diagram of type $A_{n}^{(1)}$ sending $i$ to $-i$ for all $i \in I$. It defines an automorphism $\iota_\Hlie$ of $\Hlie$ by sending $h_i, d$ to $h_{-i}, d$. Furthermore it induces an algebra automorphism of $\U_q(sl_{n+1}^{tor})$ we still denote $\iota$, which sends the series $x_{i}^{\pm}(z)$ and $\phi_i^{\pm}(z)$ to $x_{-i}^{\pm}(z)$ and $\phi_{-i}^{\pm}(z)$ respectively, and $ k_h$ to $k_{\iota_\Hlie(h)}$.

\subsubsection{Subalgebras}
There is an algebra morphism $\U_q(\hat{sl}_{n+1})\rightarrow \U_q(sl_{n+1}^{tor})$ defined by $k_h\mapsto k_h$ , $x_i^{\pm}\mapsto x_{i,0}^{\pm}$. Its image is called the horizontal quantum affine subalgebra and is denoted by $\U_q^{h}(sl_{n+1}^{tor})$. In particular, a $\U_q(sl_{n+1}^{tor})$-module $V$ has also a structure of $\U_q(\hat{sl}_{n+1})$-module. We denote by $\mathrm{Res}(V)$ the restricted $\U_q(\hat{sl}_{n+1})$-module obtained from $V$.

The quantum loop algebra $\U_q(\hat{sl}_{n+1})'$ has another realization in terms of Drinfeld generators \cite{beck_braid_1994, drinfeld_new_1988}: this is the $\CC$-algebra with generators $x_{i,r}^{\pm}$, $k_h$, $h_{i,m}$ with $i \in I_0, r \in \ZZ, m \in \ZZ-\{0\}$, $h$ an element of the Cartan subalgebra of $sl_{n+1}$, and the same relations as in Definition \ref{defqta}. It is isomorphic to the subalgebra $ \U_{q}^{v}(sl_{n+1}^{tor}) $ of $\U_q(sl_{n+1}^{tor})$ generated by the $x_{i,r}^{\pm}$, $k_{h}$, $h_{i,m}$ with $i\in I_{0}, r\in\ZZ, h \in \bigoplus_{i \in I_0} \QQ h_i, m\in\ZZ-\{0\}$.  $ \U_{q}^{v}(sl_{n+1}^{tor}) $ is called the vertical quantum affine subalgebra of $\U_q(sl_{n+1}^{tor})$.

For $J \subset I$ denote by $ \U_{q}(sl_{n+1}^{tor})_J$ the subalgebra of $ \U_{q}(sl_{n+1}^{tor}) $ generated by the $x_{i,r}^{\pm}$, $k_{h}$, $h_{i,m}$  with $i\in J, r\in\ZZ, h \in \oplus_{i \in J} \QQ h_i, m\in\ZZ-\{0\}$.

For all $j \in I$, set $I_{j} = I - \{j\}$. The subalgebra $ \U_{q}(sl_{n+1}^{tor})_{I_{j}}$ will be simply denoted $ \U_{q}^{v,j}(sl_{n+1}^{tor}) $. In particular $ \U_{q}^{v,0}(sl_{n+1}^{tor}) $ is the vertical quantum affine subalgebra $ \U_{q}^{v}(sl_{n+1}^{tor}) $ of $\U_q(sl_{n+1}^{tor})$. All the $\U_{q}^{v,j}(sl_{n+1}^{tor})$ for various $j \in I$ are isomorphic.

For $i\in I$, we will denote by $ \hat{\U}_i $ the subalgebra $ \U_{q}(sl_{n+1}^{tor})_{\{i\}}$. It is isomorphic to $\U_{q}(\hat{sl}_{2})'$.

We have a triangular decomposition of $\U_q(sl_{n+1}^{tor})$.

\begin{thm}\label{dtrian} \cite{miki_representations_2000, nakajima_quiver_2001} We have an isomorphism of vector spaces
$$\U_q(sl_{n+1}^{tor})\simeq \U_q(sl_{n+1}^{tor})^-\otimes\U_q(\hat{\Hlie})\otimes\U_q(sl_{n+1}^{tor})^+,$$ 
where $\U_q(sl_{n+1}^{tor})^{\pm}$ (resp. $\U_q(\hat{\Hlie})$) is generated by the $x_{i,r}^{\pm}$ (resp. the $k_h$, the $h_{i,m}$).
\end{thm}

\subsection{Representations of $\U_q(sl_{n+1}^{tor})$}

\subsubsection{Loop highest weight modules}

\begin{defi} A representation $V$ of $\U_q(sl_{n+1}^{tor})$ is said to be integrable (resp. in the category $\mathcal{O}$) if $\mathrm{Res}(V)$ is integrable (resp. in the category $\mathcal{O}$) as a $\U_q(\hat{sl}_{n+1})$-module.
\end{defi}

\begin{defi} A representation $V$ of $\U_q(sl_{n+1}^{tor})$ is said to be of loop highest ($\ell$-highest for short) weight if there is $v\in V$ such that 
\begin{enumerate}
\item[(i)] $V = \U_q(sl_{n+1}^{tor})^- \cdot v$, 
\item[(ii)] $\U_q(\hat{\Hlie}) \cdot v=\CC v$, 
\item[(iii)] for any $i\in I, r\in\ZZ$, $x_{i,r}^+ \cdot v=0$.
\end{enumerate}
\end{defi}

Such a representation is in the category $\mathcal{O}$. For $\gamma\in \mathrm{Hom}(\U_q(\hat{\Hlie}), \CC)$ an algebra morphism, we have a corresponding Verma module $M(\gamma)$ by Theorem \ref{dtrian}, and a simple representation $V(\gamma)$ which are $\ell$-highest weight.

\begin{thm}\label{cond} \cite{miki_representations_2000, nakajima_quiver_2001} The simple representation $V(\gamma)$ is integrable if there is $(\lambda, (P_i)_{i \in I})\in P^+ \times (1+u\mathbb{C}[u])^{I}$ satisfying $\gamma(k_h) = q^{\lambda(h)}$ for any $h \in \Hlie$ and for $i\in I$ the relation in $\mathbb{C}[[z]]$ (resp. in $\mathbb{C}[[z^{-1}]]$)
$$\gamma(\phi_i^\pm(z))=q^{\text{deg}(P_i)}\frac{P_i(zq^{-1})}{P_i(zq)}.$$
\end{thm}

\noindent The polynomials $P_i$ are called Drinfeld polynomials and the representation $V(\gamma)$ is then denoted by $V(\lambda, (P_i)_{i \in I})$.

The Kirillov-Reshetikhin module associated to $k \geq 0$, $a \in \CC^{\ast}$ and $\ell \in I$, is the simple integrable representation of weight $k \Lambda_\ell$ with the $n$--tuple
$$P_i(u) = \left\lbrace \begin{array}{l} (1-ua)(1-uaq^{2}) \cdots (1-uaq^{2(k-1)}) \ \mathrm{for} \ i = \ell, \\ 1 \ \mathrm{for} \ i \neq \ell. \end{array} \right.$$
\noindent If $k = 1$, it is also called fundamental module.

\subsubsection{Loop weight spaces}

Consider an integrable representation $V$ of $\U_q(sl_{n+1}^{tor})$. As the subalgebra $\U_q(\hat{\Hlie})$ is commutative, we have a decomposition of the weight spaces $V_{\nu}$ in simultaneous generalized eigenspaces
$$V_\nu = \bigoplus_{(\nu, \gamma) \in P \times \mathrm{Hom}(\U_q(\hat{\Hlie}), \CC)} V_{(\nu, \gamma)},$$
where $V_{(\nu, \gamma)} = \{x \in V / \exists p \in \NN, \forall i \in I, \forall m \geq 0, (\phi_{i, \pm m}^{\pm} - \gamma(\phi_{i, \pm m}^{\pm}))^{p} \cdot x = 0 \}$. If $V_{(\nu, \gamma)} \neq \{0\}$, $(\nu, \gamma)$ is called an $\ell$-weight of $V$.

\begin{defi}
A  representation $V$ is said to be $\ell$-integrable if
\begin{enumerate}
\item[(i)] it admits an $\ell$-weight space decomposition $V = \bigoplus_{(\nu, \gamma)} V_{(\nu, \gamma)}$,
\item[(ii)] for all $v \in V$ and $i\in I$, $\hat{\U}_i \cdot v$ is finite-dimensional.
\end{enumerate}
\end{defi}

All the simple $\ell$-highest weight representations $V(\lambda, (P_i)_{i \in I})$ are $\ell$-integrable. An $\ell$-integrable representation of $\U_q(sl_{n+1}^{tor})$ is integrable.

\begin{defi}
The $\U_q(sl_{n+1}^{tor})$-module $V$ is weighted if the Cartan subalgebra $\U_q(\hat{\Hlie})$ acts on $V$ by diagonalizable operators. The module $V$ is thin if it is weighted and the joint spectrum is simple.
\end{defi}

As in \cite{feigin_representations_2013}, we set $\psi(z)=\dfrac{q-q^{-1}z}{1-z}$.

\begin{prop}\cite{frenkel_$q$-characters_1999, hernandez_representations_2005, nakajima_quiver_2001}
Let $V$ be an integrable representation of $\U_q(sl_{n+1}^{tor})$. An $\ell$-weight $(\nu, \gamma)$ of $V$ satisfies the property
\begin{enumerate}
\item[(i)] for all $i \in I$, there exist $k,l \in \NN$ and $a_{1,i}, \cdots, a_{k,i}, b_{1,i}, \cdots, b_{l,i} \in \CC^{\ast}$ such that $\nu(h_i) = k-l$ and in $\CC[[z]]$ (resp. in $\CC[[z^{-1}]]$)
\begin{align}\label{formulelw}
\gamma(\phi_{i}^{\pm}(z)) = \prod_{1 \leq u \leq k} \psi(a_{u,i}qz) \prod_{1 \leq v \leq l} \psi(b_{v,i}qz)^{-1}.
\end{align}
\end{enumerate}
Further if $V$ is in the category $\mathcal{O}$, then
\begin{enumerate}
\item[(ii)] there exist $\omega \in P^{+}$, $\alpha \in Q^{+}$ satisfying $\nu = \omega - \alpha$.
\end{enumerate}
\end{prop}

\subsubsection{$q$--character of integrable representations}

Consider formal variables $ Y_{i,a}^{\pm 1} $, $ e^{\nu} $ with $ i \in I $, $ a \in \CC^{\ast}$, $ \nu \in P $, and let $ A $ be the group of monomials $ m = e^{\omega(m)} \prod_{i \in I, a \in \CC^{\ast}} Y_{i,a}^{u_{i,a}(m)} $ where $ u_{i,a}(m) \in \ZZ $, $ \omega(m) \in P $ are such that
$$\sum_{a \in \CC^{\ast}} u_{i,a}(m) = \omega(m)(h_i) \text{ for all } i \in I.$$

For $ J \subset I $, a monomial $ m $ is said to be $J$-dominant if for all $ j \in J $ and $a\in \CC^{\ast} $ we have $ u_{j,a}(m) \geq 0 $. An $ I $-dominant monomial is said to be dominant.

Let $V$ be an integrable $\U_q(sl_{n+1}^{tor})$-module. For $(\nu, \gamma)$ an $\ell$-weight of $V$, one defines from (\ref{formulelw}) the monomial $$m_{(\nu, \gamma)} = e^{\nu} \prod_{i \in I} \prod_{a_i \in \CC^{\ast}} Y_{i,a_i} \prod_{b_i \in \CC^{\ast}} Y_{i,b_i}^{-1}.$$ We denote $V_{(\nu, \gamma)} = V_{m_{(\nu, \gamma)}}$. By this correspondence between $\ell$-weights and monomials due to Frenkel-Reshetikhin \cite{frenkel_$q$-characters_1999}, the $I$-tuple of Drinfeld polynomials are identified with the dominant monomials. In particular for a dominant monomial $m$, one denotes by $V(m)$ the simple module of $\ell$-highest weight $m$. For example $V(e^{k \Lambda_\ell}Y_{\ell, a}Y_{\ell, aq^{2}} \dots Y_{\ell, aq^{2(k-1)}})$ is the Kirillov-Reshetikhin module associated to $k \geq 0$, $a \in \CC^{\ast}$ and $\ell \in I$, and $V(e^{\Lambda_\ell}Y_{\ell, a})$ is the fundamental module associated to $a \in \CC^{\ast}$ and $\ell \in I$.

\begin{defi}\cite{frenkel_$q$-characters_1999, hernandez_representations_2005, nakajima_quiver_2001} The $q$--character of an integrable representation $V$ of $\U_q(sl_{n+1}^{tor})$ with finite-dimensional $\ell$-weight spaces is defined by the formal sum
$$\chi_q(V) = \sum_{m} \dim(V_{m}) m \in \ZZ^A.$$
\end{defi}

\noindent The set of monomials occurring in $\chi_q(V)$ is denoted by $\mathcal{M}(V)$.

\subsection{Extremal loop weight modules}

\subsubsection{Definition}

\begin{defi}\label{defelm}\cite{hernandez_quantum_2009} An extremal loop weight module of $ \mathcal{U}_{q}(sl_{n+1}^{tor}) $ of $\ell$-weight $m \in A$ is an integrable representation $ V $ such that there is $ v \in V_m $ satisfying
\begin{enumerate}
\item[(i)] $ \mathcal{U}_{q}(sl_{n+1}^{tor}) \cdot v = V $,
\item[(ii)] $ v $ is extremal for $ \mathcal{U}_{q}^{h}(sl_{n+1}^{tor}) $,
\item[(iii)] for each $j \in I$ the action of $ \mathcal{U}_{q}^{v, j}(sl_{n+1}^{tor})$ is locally finite-dimensional on $V$: for any $w \in V$ the vector space $\mathcal{U}_{q}^{v, j}(sl_{n+1}^{tor}) \cdot w$ is finite-dimensional.
\end{enumerate}
\end{defi}

\subsubsection{Examples}

\begin{prop}\label{delm}
Let $m$ be a dominant monomial. The simple $\ell$-highest weight module $V(m)$ is an extremal loop weight module of $\U_q(sl_{n+1}^{tor})$.
\end{prop}

\begin{proof}
Let $v$ be a $\ell$-highest weight vector of $V(m)$. It is in particular a highest weight vector of weight $\omega(m)$, and satisfies for any $i \in I$
$$(x_i^{-})^{1+\omega(m)(h_i)} \cdot v = 0.$$
Hence we have a surjective $\U_q(\hat{sl}_{n+1})$-morphism $V(\omega(m)) \twoheadrightarrow \U_q^{h}(sl_{n+1}^{tor}) \cdot v$. As $V(\omega(m))$ is simple, this is an isomorphism and $v$ is extremal for the horizontal quantum affine subalgebra of weight $\omega(m)$.

It remains to show that for any $w \in V(m)$ and $j \in I$, $\U_q^{v,j}(sl_{n+1}^{tor}) \cdot w$ is finite-dimensional. Actually one can assume that $w$ is a weight vector of weight $\lambda$ and that $j=0$ in our computations (the third point of Definition \ref{defelm} then follows by applying $\theta^{(j)}$). Consider $W = \U_q^{v}(sl_{n+1}^{tor}) \cdot w$. As the weight lattice of $\U_q^{v}(sl_{n+1}^{tor})$ is $P_0 = \ZZ \Lambda_1 + \cdots + \ZZ \Lambda_n$, we have
$$\wt(W) \subseteq \{\mu \in \wt(V(m)) \mid \mu(d) = \lambda(d) \}.$$
As $V(m)$ is in the category $\mathcal{O}$, this set of weights is finite. The weight spaces of $V(m)$ (and as well $W$) being finite-dimensional, $W$ is then finite-dimensional.
\end{proof}

We have recently constructed the first family of extremal loop weight modules which are not in the category $\mathcal{O}$, using monomial realizations of extremal crystals.

Let $d : I \times I \rightarrow \NN$ be the map such that for all $0 \leq i,j \leq n$, $d(i,j)$ is the non-negative integer defined by
$$0 \leq d(i,j) \leq \left[ \dfrac{n+1}{2} \right] \quad \text{and} \quad \vert i - j \vert = d(i, j) \mod (n+1)$$
where $\left[ x \right]$ is the floor of the real number $x$. The integer $d(i,j)$ corresponds to the distance between the nodes $i$ and $j$ in the Dynkin diagram of type $A_n^{(1)}$. For $\ell \in I_0$, we denote $d_\ell = d(0, \ell)$.

\begin{thm}\cite{mansuy_quantum_2012}
Assume that $n=2r+1$ is odd and $\ell = 1, r+1$ or $n$, and set $a \in \CC^{\ast}$. There exists an irreducible and thin extremal loop weight $\U_q(sl_{n+1}^{tor})$-module of $\ell$-weight $e^{\varpi_\ell}Y_{\ell,a}Y_{0, aq^{d_\ell}}^{-1}$.
\end{thm}

\noindent This representation is denoted $V(e^{\varpi_\ell}Y_{\ell,a}Y_{0, aq^{d_\ell}}^{-1})$ and called extremal fundamental loop weight module\footnote{In fact in \cite{mansuy_quantum_2012} we constructed the representation $V(e^{\varpi_\ell}Y_{\ell,1}Y_{0, q^{d_\ell}}^{-1})$. The $\U_q(sl_{n+1}^{tor})$-module $V(e^{\varpi_\ell}Y_{\ell,a}Y_{0, aq^{d_\ell}}^{-1})$ is obtained from $V(e^{\varpi_\ell}Y_{\ell,1}Y_{0, q^{d_\ell}}^{-1})$ by twisting the action by $t_a$}.

Actually one can construct the representations $V(e^{\varpi_\ell}Y_{\ell,a}Y_{0, aq}^{-1})$ when $n$ is even and $\ell = 1$ or $\ell = n$ as in \cite{mansuy_quantum_2012} (by using monomial realizations of $\mathcal{B}(\varpi_\ell)$ and formulas in \cite[Theorem 4.11]{mansuy_quantum_2012}).

For $k \geq 3$ and $\epsilon \in \CC^{\ast}$ a primitive $k$-root of unity, denote by $\U_\epsilon(sl_{n+1}^{tor})'$ the algebra defined as $\U_q(sl_{n+1}^{tor})'$ with $\epsilon$ instead of $q$ (without divided power and derivation element).

\begin{prop}\cite{mansuy_quantum_2012}
Set $p = n+1$ and $L \geq 1$ when $\ell = 1, n$ (resp. $p = 2$ and $L>1$ when $n$ is odd and $\ell = \frac{n+1}{2}$). Denote by $\epsilon$ a primitive $[pL]$-root of unity.

\noindent The $\U_\epsilon(sl_{n+1}^{tor})'$-module $V(e^{\varpi_\ell}Y_{\ell,a}Y_{0, aq^{d_\ell}}^{-1})_\epsilon$ obtained from $V(e^{\varpi_\ell}Y_{\ell,a}Y_{0, aq^{d_\ell}}^{-1})$ by specializing $q$ at $\epsilon$ admits an irreducible and finite-dimensional quotient.
\end{prop}

\begin{prop}\cite{mansuy_quantum_2012}
\begin{enumerate} 
\item[(i)]There exists an irreducible and thin extremal loop weight $\U_q(sl_{4}^{tor})$-module of $\ell$-weight $e^{2\varpi_1}Y_{1,a}Y_{1,aq^{-2}}Y_{0,aq}^{-1}Y_{0,aq^{-1}}^{-1}$.

\noindent We denote this representation $V(e^{2\varpi_1}Y_{1,a}Y_{1,aq^{-2}}Y_{0,aq}^{-1}Y_{0,aq^{-1}}^{-1})$\footnote{It is obtained from the representation $V(e^{2\varpi_1}Y_{1,q}Y_{1,q^{-1}}Y_{0,q^{2}}^{-1}Y_{0,1}^{-1})$ constructed in \cite{mansuy_quantum_2012} by twisting the action by $t_{aq^{-1}}$}.
\item[(ii)] Set $\epsilon$ a primitive $[4L]$-root of unity, with $L \geq 1$.

\noindent The $\U_\epsilon(sl_{4}^{tor})'$-module obtained from $V(e^{2\varpi_1}Y_{1,q}Y_{1,q^{-1}}Y_{0,q^{2}}^{-1}Y_{0,1}^{-1})$ by specializing $q$ at $\epsilon$ admits an irreducible and finite-dimensional quotient.
\end{enumerate}
\end{prop}


\subsection{Coproduct}

\subsubsection{Drinfeld coproduct} Let $\Delta_D$ be the coproduct defined for all $i\in I$ by
\begin{equation}\label{deltaxplus}
\Delta_D(x_i^+(z))=x_i^+(z)\otimes 1 +\phi_i^-(z)\otimes x_i^+(z),
\end{equation}
\begin{equation}\label{deltaxmoins}
\Delta_D(x_i^-(z))= x_i^-(z)\otimes \phi_i^+(z)+1\otimes x_i^-(z),
\end{equation}
\begin{equation}\label{deltacart}
\Delta_D(\phi_i^\pm(z))=\phi_i^\pm(z)\otimes \phi_i^\pm (z).
\end{equation}

\noindent This map does not define a coproduct in the usual sense because (\ref{deltaxplus}) and (\ref{deltaxmoins}) involve infinite sums and are not elements of $\U_q(sl_{n+1}^{tor}) \otimes \U_q(sl_{n+1}^{tor})$. However it can be used to define a structure of $\U_q(sl_{n+1}^{tor})$-module on (a subspace of) a tensor product of $\U_q(sl_{n+1}^{tor})$-modules when the summations are well-defined.

\subsubsection{Tensor product of representations}

\begin{lem}\label{lemsubm}\cite{feigin_representations_2013}
Let $V$, $W$ be $\U_q(sl_{n+1}^{tor})$-modules. Let $U\subseteq V\otimes W$ be a linear subspace such that 
for any $u\in U$ and $i\in I$, $\Delta_D(x_{i}^{\pm}(z)) \cdot u$ is well-defined in $U$. Then $\Delta_D$ endows $U$ with a structure of $\U_q(sl_{n+1}^{tor})$-module.
\end{lem}

One can define a structure of $\U_q(sl_{n+1}^{tor})$-module in another situation: assume that we have a decomposition of vector spaces $V \otimes W = U \oplus U'$ such that $\Delta_D(\phi_i^\pm(z)) \cdot U \subset U$ and $\Delta_D(\phi_i^\pm(z)) \cdot U' \subset U'$ for all $i \in I$. Let us denote by $\mathrm{pr}_U : U \oplus U' \rightarrow U$ and $\mathrm{pr}_{U'} : U \oplus U' \rightarrow U'$ the projections along $U$ and $U'$ respectively. Assume that the action of $\Delta_D(x_i^{\pm}(z))$ is well-defined on the vectors in $U$. Assume further that the operators $\mathrm{pr}_U \circ \Delta_D(x_i^{\pm}(z))$ applied to vectors in $U'$ are zero. Then we have

\begin{lem}\label{lemquo} \cite{feigin_representations_2013}
Under the assumptions above, the space $U$ has a structure of $\U_q(sl_{n+1}^{tor})$-module such that the series $x_i^{\pm}(z)$ and $\phi_i^\pm(z)$ act by $\mathrm{pr}_U \circ \Delta_D(x_i^{\pm}(z))$ and $\mathrm{pr}_U \circ \Delta_D(\phi_i^\pm(z))$ respectively.
\end{lem}

\begin{lem}\label{lemsubm2}
Let $V_1, \cdots , V_k$ be thin $\U_q(sl_{n+1}^{tor})$-modules. Assume that for any $i \in \ZZ$, $1 \leq r < s \leq k$ and $v \otimes w \in V_r \otimes V_s$, $\Delta_D(x_i^\pm(z)) \cdot (v \otimes w)$ is a well-defined vector in $V_r \otimes V_s$. Then $\Delta_D$ endows $V_1 \otimes V_2 \otimes \cdots \otimes V_k$ with a structure of $\U_q(sl_{n+1}^{tor})$-module.
\end{lem}

\begin{lem}\label{tpalgv}
Let $V_1, \cdots , V_k$ be $\U_q(sl_{n+1}^{tor})$-modules such that for $j \in I$ the action of $\U_q^{v,j}(sl_{n+1}^{tor})$ is locally finite-dimensional on all the $V_r$. Assume that $\Delta_D$ endows a linear subspace $U \subseteq V_1 \otimes V_2 \otimes \cdots \otimes V_k$ with a structure of $\U_q(sl_{n+1}^{tor})$-module. Then the action of $\U_q^{v,j}(sl_{n+1}^{tor})$ on $U$ is also locally finite-dimensional.

If moreover the $V_r$ are $\ell$-integrable, then $U$ is an integrable $\U_q(\hat{sl}_\infty)$-module.
\end{lem}

\noindent The proofs of Lemma \ref{lemsubm2} and Lemma \ref{tpalgv} are straightforward.

\subsubsection{Deformation of the Drinfeld coproduct and $\mathcal{A}$-forms}

Consider the principal ring
$$\mathcal{A} = \{ f(u) \in \CC(u) \mid f \text{ regular at } u=1\} \subset \CC(u).$$
Let $\U_q^u(sl_{n+1}^{tor}) = \U_q(sl_{n+1}^{tor}) \otimes_{\CC} \mathcal{A}$. The $u$-deformation $\Delta_D^u$ of the Drinfeld coproduct is defined in \cite{hernandez_drinfeld_2007} as the $\CC(u)$-algebra morphism such that
$$\Delta_D^u : \U_q(sl_{n+1}^{tor}) \otimes_{\CC} \CC(u) \rightarrow (\U_q(sl_{n+1}^{tor}) \otimes \U_q(sl_{n+1}^{tor}))((u)),$$
$$\Delta_D^u(x) = (\mathrm{Id} \otimes t_u) \circ \Delta_D (x) \text{ for any }x \in \U_q(sl_{n+1}^{tor}).$$

\begin{defi}\cite{hernandez_drinfeld_2007}
Let $V$ be an integrable $\U_q(sl_{n+1}^{tor}) \otimes_{\CC} \CC(u)$-module. An $\mathcal{A}$-form $\tilde{V}$ of $V$ is a sub $\mathcal{A}$-module of $V$ satisfying
\begin{itemize}
\item[(i)] $\tilde{V}$ is stable under $\U_q^u(sl_{n+1}^{tor})$,
\item[(ii)] $\tilde{V}$ generates the $\CC(u)$-vector space $V$,
\item[(iii)] for $m \in A$,  $\tilde{V} \cup V_m$ is a finitely generated $\mathcal{A}$-module.
\end{itemize}
\end{defi}

\begin{prop}\cite{hernandez_drinfeld_2007}
Let $V$ be an integrable $\U_q(sl_{n+1}^{tor}) \otimes_{\CC} \CC(u)$-module and $\tilde{V}$ an $\mathcal{A}$-form of $V$. Then $\tilde{V}/((u-1)\tilde{V})$ is an integrable $\U_q(sl_{n+1}^{tor})$-module.
\end{prop}

\section{Tensor product of $\ell$-highest weight modules and $\ell$-lowest weight modules}

In this section we study tensor products of a simple $\ell$-highest weight module and a simple $\ell$-lowest weight module. The main motivation comes from Remark \ref{remkash}: as for the extremal weight modules, we will see that such tensor products are related to extremal loop weight representations. Let us mention some technical points: to endow these tensor products with a structure of $\U_q(sl_{n+1}^{tor})$-module, we have to prove that the summations (\ref{deltaxplus}) and (\ref{deltaxmoins}) are well-defined on them. In fact their convergence depends on the non-zero complex parameters defining the $\ell$-highest weight and the $\ell$-lowest weight. We made the choice to consider in the article the following tensor product
$$V(e^{s \Lambda_\ell}Y_{\ell, a}^{s}) \otimes V(e^{-s \Lambda_0}Y_{0, b}^{-s}) \text{ with } 1 \leq \ell \leq n, \ a, b \in \CC^{\ast} \text{ and } s \geq 1.$$

In the first part of this section we prove that, under a condition on $a$ and $b$, $\Delta_D$ endows this tensor product with a structure of $\U_q(sl_{n+1}^{tor})$-module (Theorem \ref{tpexist}). We study the link with the extremal loop weight representations in the second part (Proposition \ref{propnecon}). We get in this way a new construction of the vector representation (Theorem \ref{repvect2}).

\subsection{Existence of the action on the tensor product}

\subsubsection{Main result}

\begin{thm}\label{tpexist}
Let $\ell \in I_0, s \in \NN^{\ast}$ and $a, b \in \CC^{\ast}$ be non zero complex numbers and consider the tensor product
$$V(e^{s \Lambda_\ell}Y_{\ell, a}^{s}) \otimes V(e^{-s \Lambda_0}Y_{0, b}^{-s}).$$
Assume that $\dfrac{a}{b} \notin q^{-2-d_\ell-\NN}$. Then the coproduct $\Delta_D$ is well-defined on this tensor product and endows it with a structure of $\U_q(sl_{n+1}^{tor})$-module.
\end{thm}

\noindent By Lemma \ref{tpalgv}, the tensor product module hence obtained is integrable and the action of $\U_q^{v,j}(sl_{n+1}^{tor})$ on it is locally finite-dimensional for all $j \in I$. The rest of this section is devoted to the proof of Theorem \ref{tpexist}.

\subsubsection{Intermediary results}

\begin{prop}\label{actxphi} \cite{hernandez_drinfeld_2007}
Let $V$ be a finite-dimensional $\U_q(\hat{sl}_2)'$-module and fix $m \in \ZZ$. There are finite numbers of $\CC$-linear operators $A_{k, \lambda_\pm}^{\pm}, B_{k, \lambda_\pm}^{\pm}, C_{k, \lambda_\pm}^{\pm} : V \rightarrow V $ ($k \geq 0$, $ \lambda_\pm \in \CC^{\ast}$) such that for any $r \geq 0$, $s \geq 1$ and $v \in V$,

\begin{eqnarray}\label{actsl2}
\begin{array}{rcl}
x_{1,m \pm r}^{+} \cdot v &=& \sum_{k \geq 0, \lambda_\pm \in \CC^{\ast}} \lambda_\pm^r r^k A_{k, \lambda_\pm}^{\pm}(v),\\
 & & \\
 x_{1,m \pm r}^{-} \cdot v &=& \sum_{k \geq 0, \lambda_\pm \in \CC^{\ast}} \lambda_\pm^r r^k B_{k, \lambda_\pm}^{\pm}(v),\\
 & & \\
\phi_{1,\pm s}^{\pm} \cdot v &=& \sum_{k \geq 0, \lambda_\pm \in \CC^{\ast}} \lambda_\pm^s s^k C_{k, \lambda_\pm}^{\pm}(v).
\end{array}
\end{eqnarray}

\noindent Furthermore the non-zero complex numbers $\lambda_{\pm}$ are the eigenvalues of the operators
$$\Phi_{\pm} : \mathrm{End}(V) \rightarrow \mathrm{End}(V), \quad \Phi_{\pm} = \frac{1}{[2]_q} \mathrm{ad}(h_{1,\pm 1}).$$
\end{prop}

\noindent For $m = Y_{1,a_1} \cdots Y_{1,a_k} Y_{1,b_1}^{-1} \cdots Y_{1,b_l}^{-1} \in A$, we associate complex numbers $\lambda_\pm(m)$ such that
$$\lambda_\pm(m) = \sum_{1 \leq u \leq k} a_u^{\pm 1} - \sum_{1 \leq v \leq l} b_v^{\pm 1}.$$

\begin{lem}
The eigenvalues $\lambda_\pm$ of $\Phi_{\pm}$ are of the form
$$\lambda_\pm = \dfrac{\lambda_\pm(m) - \lambda_\pm(m')}{[2]_q}$$
with $m, m' \in \mathcal{M}(V)$.
\end{lem}

\begin{proof}
By a standard result of Lie algebras an eigenvalue of $\mathrm{ad}(h_{1,\pm 1})$ is of the form $\alpha_\pm - \beta_\pm$ with $\alpha_\pm, \beta_\pm$ eigenvalues of $h_{1, \pm 1}$.

Let $m \in \mathcal{M}(V)$. It is well-known (see \cite{frenkel_$q$-characters_1999}) that the $\ell$-weight space $V_m$ is an eigenspace for the operators $h_{1,\pm 1}$, associated to the eigenvalues $\lambda_\pm(m)$. Hence the complex numbers $\lambda_\pm$ we have to consider are of the form
$$\lambda_\pm = \dfrac{\lambda_\pm(m) - \lambda_\pm(m')}{[2]_q} \quad \text{with} \quad m, m' \in \mathcal{M}(V).$$
\end{proof}

Consider the $\U_q(sl_{n+1}^{tor})$-module $V(e^{s \Lambda_\ell}Y_{\ell, a}^{s})$ with $1 \leq \ell \leq n$, $a \in \CC^{\ast}$ and $s \in \NN^{\ast}$. For $0 \leq i \leq n$ and $w \in V(e^{s \Lambda_\ell}Y_{\ell, a}^{s})$, set $W = \hat{\U}_i \cdot w$.

\begin{lem}\label{indvp}
For all $m \in \mathcal{M}(W)$, we have
$$\lambda_+(m) \in aq^{d(i, \ell)} \cdot ( \ZZ + q \ZZ[q]) \quad \text{ and } \quad \lambda_-(m) \in a^{- 1}q^{- d(i, \ell)} \cdot ( \ZZ + q^{- 1} \ZZ[q^{- 1}]).$$
\end{lem}

\begin{proof}
The theory of $q$--characters for the simple $\ell$-highest weight $\U_q(sl_{n+1}^{tor})$-modules (see \cite{frenkel_$q$-characters_1999, hernandez_representations_2005, hernandez_drinfeld_2007, nakajima_quiver_2001}) shows that the smallest integer $k$ such that variable $Y_{i,aq^{k}}$ occurs in $\chi_q(V(e^{s\Lambda_\ell}Y_{\ell,a}^{s}))$ is for $k = d(i, \ell)$. The result follows directly.
\end{proof}

\begin{rem}
An analogous result can be stated for the $\U_q(sl_{n+1}^{tor})$-module $V(e^{-s \Lambda_0}Y_{0, b}^{-s})$: set $W = \hat{\U}_i \cdot v$ with $0 \leq i \leq n$ and $v \in V(e^{-s \Lambda_0}Y_{0, b}^{-s})$. We have for any $m \in \mathcal{M}(W)$,
$$\lambda_+(m) \in bq^{- d(i, 0)} \cdot ( \ZZ + q^{-1} \ZZ[q^{-1}]) \quad \text{ and } \quad \lambda_-(m) \in b^{- 1}q^{d(i, 0)} \cdot ( \ZZ + q \ZZ[q]).$$
\end{rem}

\subsubsection{Proof of Theorem \ref{tpexist}}

Let us show that
$$\tilde{V} = \left( V(e^{s \Lambda_\ell}Y_{\ell, a}^{s}) \otimes_{\CC} V(e^{-s \Lambda_0}Y_{0, b}^{-s}) \right) \otimes_{\CC} \mathcal{A}$$
is an $\mathcal{A}$-form of $V = \left(V(e^{s \Lambda_\ell}Y_{\ell, a}^{s}) \otimes_{\CC} V(e^{-s \Lambda_0}Y_{0, b}^{-s}) \right) \otimes_{\CC} \CC(u) $. For $i \in I$ and $v \otimes w \in \tilde{V}$, the vector $\Delta_D^u (x_i^+(z)) \cdot (v \otimes w)$ is equal to
\begin{eqnarray*}
(x_i^+(z) \cdot v) \otimes w + (\phi_i^-(z) \cdot v) \otimes (x_i^+(uz) \cdot w)
\end{eqnarray*}
For $m \in \ZZ$, the coefficient of $z^m$ in the last term above is
\begin{eqnarray*}
\sum_{r \geq 0}
u^{m+r}(\phi_{i,-r}^- \cdot v) \otimes (x_{i,m+r}^+ \cdot w) &=& u^m(k_i^{-1} \cdot v) \otimes (x_{i,m}^+ \cdot w) \\
& & + \sum_{r \geq 1} u^{m+r}(\phi_{i,-r}^- \cdot v) \otimes (x_{i,m+r}^+ \cdot w).
\end{eqnarray*}
With notations used above, the infinite sum at the right hand side is equal to
\begin{eqnarray}\label{infsum}
\sum_{k \geq 0, \lambda_- \in \CC^{\ast}} \sum_{l \geq 0, \lambda_+ \in \CC^{\ast}} \underset{=: \gamma_{k,l, \lambda_-, \lambda_+}(u)}{\underbrace{\left( \sum_{r \geq 1} \lambda_-^r \lambda_+^r r^{k+l} u^{m+r} \right)}} C_{k, \lambda_-}^-(v) \otimes A_{l, \lambda_+}^+(w).
\end{eqnarray}
The $\gamma_{k,l, \lambda_-, \lambda_+}(u)$ are rational fractions on $u$ with pole at $u = (\lambda_- \cdot \lambda_+)^{-1}$. They are regular at $u = 1$: indeed for all $\lambda_\pm$ occurring in (\ref{infsum}) we have by Lemma \ref{indvp}
$$[2]_q \cdot \lambda_- \in a^{- 1}q^{- d(i, \ell)} \cdot ( \ZZ + q^{- 1} \ZZ[q^{- 1}])$$
and
$$[2]_q \cdot \lambda_+ \in bq^{- d(i, 0)} \cdot ( \ZZ + q^{-1} \ZZ[q^{-1}]).$$
So $([2]_q)^2 \cdot \lambda_- \lambda_+$ belongs to $\dfrac{b}{a} q^{- d(i, 0)-d(i,\ell)} \cdot ( \ZZ + q^{-1} \ZZ[q^{-1}])$. As $\dfrac{b}{a} \notin q^{2+d(0,\ell)+\NN}$, $([2]_q)^2 \cdot \lambda_- \lambda_+ \neq q^{2} + 1 + q^{-2}$ and $\gamma_{k,l, \lambda_-, \lambda_+}(u) \in \mathcal{A}$. The proof is the same for the operators $x_i^-(z)$.

Hence if $\dfrac{a}{b} \notin q^{-2-d(0,\ell)-\NN}$ $\tilde{V}$ is an $\mathcal{A}$-form of $V$, and $V(e^{s \Lambda_\ell}Y_{\ell, a}^{s}) \otimes_{\CC} V(e^{-s \Lambda_0}Y_{0, b}^{-s}) = \tilde{V} / (u-1)\tilde{V}$ can be endowed with a structure of $\U_q(sl_{n+1}^{tor})$-module.

\subsubsection{Remark}

Working with $\U_{q}(sl_{n+1}^{tor})$ instead of its $d$-deformation $\U_{q,d}(sl_{n+1}^{tor})$ led us to different considerations than in \cite{feigin_representations_2013}: indeed the $\U_{q}(sl_{n+1}^{tor})$-modules (even the fundamental modules) are not weighted in general (see \cite[Section 4.1]{hernandez_quantum_2009}) and we have to deal with generalized eigenspaces. These difficulties do not occur for $\U_{q,d}(sl_{n+1}^{tor})$ for which all the simple loop highest weight $\U_{q,d}(sl_{n+1}^{tor})$-modules are weighted (see \cite{feigin_representations_2013}).


\subsection{Tensor products and extremal loop weight modules}

Assume that $\dfrac{a}{b} \notin q^{-2-d_\ell-\NN}$ in the whole section. Consider the $\U_q(sl_{n+1}^{tor})$-module
$$T(e^{\varpi_\ell}Y_{\ell, a}Y_{0, b}^{-1}) = \U_q(sl_{n+1}^{tor}) \cdot v \subseteq V(e^{\Lambda_\ell}Y_{\ell, a}) \otimes V(e^{- \Lambda_0}Y_{0, b}^{-1}).$$
Here $v $ is the vector equal to $ v^+ \otimes v^-$ where $v^+$ (resp. $v^-$) is a $\ell$-highest weight vector of $V(e^{\Lambda_\ell}Y_{\ell, a})$ (resp. a $\ell$-lowest weight vector of $V(e^{-\Lambda_0}Y_{0, b}^{-1})$).

We determine necessary conditions for $T(e^{\varpi_\ell}Y_{\ell, a}Y_{0, b}^{-1})$ to be an extremal loop weight module generated by the extremal vector $v$. Furthermore we show that these conditions are sufficient in some cases.

\subsubsection{Necessary condition}

\begin{prop}\label{propnecon}
For the vector $v \in T(e^{\varpi_\ell}Y_{\ell, a}Y_{0, b}^{-1})$ to be extremal for the horizontal quantum affine subalgebra, it is necessary that $\dfrac{a}{b} = q^{-d_\ell}$ and
\begin{align}
\begin{cases}
n \text{ is even and } \ell = 1, n \\
\text{or}\\
n = 2r+1 \text{ is odd and } \ell = 1, r+1, n.
\end{cases}
\tag{C}
\end{align}
\end{prop}

\medskip

\begin{proof}
Assume that $1 < \ell \leq \left[ \frac{n+1}{2} \right]$ (the case $\left[ \frac{n+1}{2} \right] \leq \ell < n $ deduces by applying $\iota$). Let us show that a necessary condition for $v$ to be extremal is 
$$n = 2r+1 \text{ odd, } \ell = r+1 \text{ and } \dfrac{a}{b} = q^{-\ell}.$$

\noindent By straightforward computations, we check that for all $i \in I$ and $\ell \leq t \leq n-1$, $S_{t} \circ \cdots \circ S_\ell (v)$ is $i$-extremal and
$$ S_{n} \circ \cdots \circ S_\ell (v) = \left( S_{n} \circ \cdots \circ S_\ell (v^+) \right) \otimes v^-.$$
The vector hence obtained $w_1 = S_{n} \circ \cdots \circ S_\ell (v^+)$ is of $\ell$-weight $Y_{\ell-1, aq}Y_{n,aq^{n+2-\ell}}^{-1}Y_{0,aq^{n+1-\ell}}$, and we compute that
\begin{align*}
x_{0,0}^{-} \cdot \left( w_1 \otimes v^- \right)= \psi \left(\dfrac{b}{a} q^{\ell-n-1} \right)^{-1} (x_{0,0}^{-} \cdot w_1) \otimes v^-.
\end{align*}
But the vector $w_1 \otimes v^-$ is of weight $\Lambda_{\ell-1} - \Lambda_n$. Then $v$ is extremal only if
$$\psi \left(\dfrac{b}{a} q^{\ell-n-1} \right)^{-1} = 0 \Leftrightarrow \dfrac{b}{a} = q^{n+1-\ell}.$$

\noindent We show in a symmetric way that we must have $\dfrac{b}{a} = q^{\ell}$. Hence a necessary condition for $v$ to be extremal is $\dfrac{b}{a} = q^{\ell}$ and
$q^{n+1-\ell} = q^{\ell}$, i.e. $ 2\ell = n+1$ (note that in this case $\dfrac{b}{a} \notin q^{2 + \ell + \NN}$ and the action on the tensor product is well-defined). The proof for the cases $\ell = 1, n$ is exactly the same and is not detailed.
\end{proof}

\begin{rem}
Condition $(C)$ provides to the cyclic symmetry of the Dynkin diagram of type $A_{n}^{(1)}$. For quantum affinizations associated to Dynkin diagrams without cycle, we expect to obtain extremal loop weight modules for all the nodes of the Dynkin diagram.
\end{rem}

\subsubsection{Cases $\ell = 1, n$}

\begin{prop}\label{repvect} The $\U_q(sl_{n+1}^{tor})$-module $T(e^{\varpi_1}Y_{1,a}Y_{0,aq}^{-1})$ has a basis labeled by the tableaux\footnote{To facilitate the computations to come, the spectral parameter $a$ is recalled as subscript.}
\begin{equation*}
  \ffbox{j}_a \qquad
  \text{with $j \in \ZZ$},
\end{equation*}
on which the action of $\U_q(sl_{n+1}^{tor})$ is 
\begin{eqnarray}\label{actrv}
\begin{array}{ccl}
x_{i}^{+}(z) \cdot \ffbox{j}_a &=& \delta_{i, \overline{j-1}} \delta(a q^{j-1}z) \cdot \ffbox{j-1}_a,\\
 & & \\
x_{i}^{-}(z) \cdot \ffbox{j}_a &=& \delta_{i, \overline{j}} \delta(a q^{j}z) \cdot \ffbox{j+1}_a,\\
 & & \\
\phi_{i}^{\pm}(z) \cdot \ffbox{j}_a &=& \begin{cases} \psi(aq^{j}z) \cdot \ffbox{j}_a & \text{ if } i = \overline{j},\\
\psi(aq^{j+1}z)^{-1}\cdot \ffbox{j}_a & \text{ if } i = \overline{j-1},\\
\ffbox{j}_a & \text{ otherwise}
\end{cases}
\end{array}
\end{eqnarray}
for all $i \in I$ and $j \in \ZZ$ and where $\delta_{u,v}$ is the Kronecker delta.
\end{prop}

\noindent The proof of Proposition \ref{repvect} is given in Appendix. As a direct consequence we have

\begin{thm}\label{repvect2}
Assume that $\ell = 1$ or $\ell = n$. The $\U_q(sl_{n+1}^{tor})$-module $T(e^{\varpi_\ell}Y_{\ell,a}Y_{0,aq}^{-1})$ is an extremal loop weight module of $\ell$-weight $e^{\varpi_\ell}Y_{\ell,a}Y_{0,aq}^{-1}$.

Furthermore $T(e^{\varpi_\ell}Y_{\ell,a}Y_{0,aq}^{-1})$ is isomorphic to the extremal fundamental loop weight $\U_q(sl_{n+1}^{tor})$-module $V(e^{\varpi_\ell}Y_{\ell,a}Y_{0,aq}^{-1})$.
\end{thm}

\begin{rem}
\begin{enumerate}
\item The representation $V(e^{\varpi_1}Y_{1,a}Y_{0,aq}^{-1})$ is introduced in \cite{feigin_representations_2013} for the \linebreak $d$-deformation $\U_{q,d}(sl_{n+1}^{tor})$ of the quantum toroidal algebra. This is the quantum version of the representation $\CC^{n+1} \otimes \CC[t^{\pm 1}]$ of the Lie algebra $sl_{n+1}(\CC) \otimes \CC[t^{\pm 1}, s^{\pm 1}]$. It is called vector representation, and denoted $V(\ffbox{1}_a)$ in the following.
\item The case $n = 2r+1$ odd and $\ell = r+1$ is more complicated to treat. We expect to obtain an analogous result than Theorem \ref{repvect2}. This will be studied elsewhere. The $\U_{q}(sl_{2r+2}^{tor})$-modules $V(Y_{r+1,a}Y_{0,aq^{r+1}}^{-1})$ will be recover thereafter as subquotients of tensor products of vector representations (Theorem \ref{tpeflwm}).
\end{enumerate}

\end{rem}

\section{Tensor products of vector representations}

In this section we study tensor products of vector representations $V(\ffbox{1}_a)$. We obtain in this way new families of extremal loop weight modules for $\U_q(sl_{n+1}^{tor})$. Furthermore we recover all the extremal loop weight representations constructed in \cite{mansuy_quantum_2012}.

In the first part we recall existence conditions of the action on tensor products of two vector representations obtained in \cite{feigin_representations_2013}. In the second part we study the case of tensor products of $k$ vector representations with generic parameters $a_1, \cdots, a_k \in \CC^{\ast}$ (Theorem \ref{thmexgen}). In the third part we are interested in tensor products of vector representations when the set of parameters forms a $q$-segment $\{a, aq^{-2}, \cdots, aq^{-2(k-1)}\}$ (Theorem \ref{thmexng} and Proposition \ref{propexmod}). As subquotients, we recover the extremal loop weight representations we defined in \cite{mansuy_quantum_2012} (Proposition \ref{proprecrep} and Theorem \ref{tpeflwm}). In the last part, we obtain by specialization new irreducible finite-dimensional representations of quantum toroidal algebras at roots of unity.

\subsection{Existence conditions}

\begin{prop}\label{proptpdef} \cite{feigin_representations_2013}
\begin{enumerate}
\item[(i)] The action of $\U_q(sl_{n+1}^{tor})$ on the tensor product $V(\ffbox{1}_a) \otimes V(\ffbox{1}_b)$ is well-defined if and only if $\dfrac{a}{b} \notin q^{(n+1) \ZZ}$.
\item[(ii)] Assume that $\dfrac{a}{b} = q^{- 2 + m (n+1)}$ with $m \in \ZZ$. The module $V(\ffbox{1}_a) \otimes V(\ffbox{1}_b)$ has a submodule spanned by vectors of the form $\ffbox{i}_a \otimes \ffbox{j}_b$ with $i \leq j+m(n+1)$. The submodule and the quotient module are irreducible and thin.
\item[(iii)] Assume that $\dfrac{a}{b} \in q^{2 + m(n+1)}$ with $m \in \ZZ$. The module $V(\ffbox{1}_a) \otimes V(\ffbox{1}_b)$ has a submodule spanned by vectors of the form $\ffbox{i}_a \otimes \ffbox{j}_b$ with $i < j+m(n+1)$. The submodule and the quotient module are irreducible and thin.
\item[(iv)] In all the other cases, $V(\ffbox{1}_a) \otimes V(\ffbox{1}_b)$ is a thin, irreducible $\U_q(sl_{n+1}^{tor})$-module.
\end{enumerate}
\end{prop}


In the following sections, we give new results about tensor products of vector representations in the context of extremal loop weight modules.

\subsection{The generic case}

\subsubsection{Main result}

\begin{thm}\label{thmexgen}
Let $k \in \NN^{\ast}$ and $a_1, \cdots, a_k \in \CC^{\ast}$ be such that $\frac{a_i}{a_j} \notin q^{(n+1)\ZZ}$ and \linebreak $\frac{a_i}{a_j} \notin q^{\pm 2 +(n+1)\ZZ}$ for all $i < j$. The tensor product of vector representations $$V(\ffbox{1}_{a_1}) \otimes V(\ffbox{1}_{a_2}) \otimes \cdots \otimes V(\ffbox{1}_{a_k})$$ is an irreducible and thin extremal loop weight $\U_q(sl_{n+1}^{tor})$-module of $\ell$-weight $$e^{k \varpi_1}Y_{1,a_1}\cdots Y_{1, a_k}Y_{0,qa_1}^{-1} \cdots Y_{0, qa_k}^{-1}.$$
\end{thm}

\noindent We prove the theorem in the rest of this section.

\subsubsection{Existence of the action}

\begin{lem}
The action of $\U_q(sl_{n+1}^{tor})$ on $$V = V(\ffbox{1}_{a_1}) \otimes V(\ffbox{1}_{a_2}) \otimes \cdots \otimes V(\ffbox{1}_{a_k})$$ is well-defined and the module hence obtained is thin.
\end{lem}

\begin{proof}
The existence of the action follows by Lemma \ref{lemsubm2} and Proposition \ref{proptpdef}. The thin property is straightforward, the $\ell$-weight $\prod_{i} Y_{\overline{j_i},a_iq^{j_i-1}}Y_{\overline{j_i-1},a_iq^{j_i}}^{-1}$ of a vector $\ffbox{j_1}_{a_1} \otimes \cdots \otimes \ffbox{j_k}_{a_k}$ being completely determined by the sequence $(j_1, \cdots , j_k)$.
\end{proof}

\noindent By Lemma \ref{tpalgv}, $V$ is integrable and the action of vertical quantum affine subalgebras on it is locally finite-dimensional.

\subsubsection{Proof of $V$ irreducible}

\begin{lem}\label{lemirrmod}
The $\U_q(sl_{n+1}^{tor})$-module $V$ is irreducible.
\end{lem}

\begin{proof}
Let us compute the vector $x_i^-(z) \cdot \left( \ffbox{j_1}_{a_1} \otimes \cdots \otimes \ffbox{j_k}_{a_k} \right)$: it is equal to
\begin{eqnarray}\label{formref}
\sum_{\begin{tiny} \begin{array}{c} 1 \leq u \leq k \\ \overline{j_u} = i \end{array} \end{tiny}} A_u \ \delta(a_u q^{j_u}z) \cdot \ffbox{j_1}_{a_1} \otimes \cdots \otimes \ffbox{j_u+1}_{a_u} \otimes \cdots \otimes \ffbox{j_k}_{a_k}
\end{eqnarray}
with $A_u = \displaystyle \prod_{\begin{tiny} \begin{array}{c} u < s \leq k\\ \overline{j_s} = i \end{array} \end{tiny}} \psi \left( \frac{a_s}{a_u} q^{j_s-j_u} \right) \prod_{\begin{tiny} \begin{array}{c} u < s \leq k\\ \overline{j_s} = i+1 \end{array} \end{tiny}} \psi \left(\frac{a_s}{a_u} q^{j_s-j_u+1} \right)^{-1}$.

\noindent By hypothesis on the $a_i$, we have $A_u \neq 0$ for all $u$. As $V$ is thin, for each $1 \leq u \leq k$ there exists $x \in \U_q(sl_{n+1}^{tor})$ such that
$$x \cdot \left( \ffbox{j_1}_{a_1} \otimes \cdots \otimes \ffbox{j_u}_{a_u} \otimes \cdots \otimes \ffbox{j_k}_{a_k} \right) = \ffbox{j_1}_{a_1} \otimes \cdots \otimes \ffbox{j_u + 1}_{a_u} \otimes \cdots \otimes \ffbox{j_k}_{a_k}.$$
We show in the same way that for all $1 \leq v \leq k$, there exists $y \in \U_q(sl_{n+1}^{tor})$ such that
$$y \cdot \left(\ffbox{j_1}_{a_1} \otimes \cdots \otimes \ffbox{j_v}_{a_v} \otimes \cdots \otimes \ffbox{j_k}_{a_k}\right) = \ffbox{j_1}_{a_1} \otimes \cdots \otimes \ffbox{j_v - 1}_{a_v} \otimes \cdots \otimes \ffbox{j_k}_{a_k}.$$
It follows directly that $V$ is irreducible.
\end{proof}

\subsubsection{Proof of $V$ loop extremal}

\begin{lem}\label{extvectgen}
The vector $$\ffbox{1}_{a_1} \otimes \cdots \otimes \ffbox{1}_{a_k} \in V$$ is extremal for the horizontal quantum affine subalgebra.
\end{lem}

\begin{proof}
Applying $k$ times formula (\ref{formref}), we get
\begin{eqnarray*}
(x_{1,0}^-)^{(k)} \cdot \left( \ffbox{1}_{a_1} \otimes \cdots \otimes \ffbox{1}_{a_k} \right) = \frac{1}{[k]_q} \left( \sum_{1 \leq i \leq k} \prod_{j \neq i} \psi(a_j / a_i) \right) \cdot \ffbox{2}_{a_1} \otimes \cdots \otimes \ffbox{2}_{a_k}
\end{eqnarray*}
By the equality (see Example 2.(a) on p.319 of \cite{macdonald_symmetric_1995}) into the field $\CC(X_1, \cdots, X_k, Q)$
\begin{equation*}
\sum_{1 \leq i \leq k} \prod_{j \neq i} \dfrac{Q X_i - Q^{-1}X_j}{X_i-X_j} = \dfrac{Q^{k} - Q^{-k}}{Q-Q^{-1}}= [k]_{Q},
\end{equation*}
we obtain
$$(x_{1,0}^-)^{(k)} \cdot \left( \ffbox{1}_{a_1} \otimes \cdots \otimes \ffbox{1}_{a_k} \right) = \ffbox{2}_{a_1} \otimes \cdots \otimes \ffbox{2}_{a_k}.$$
One can show in the same way that for all $ j \in \ZZ$,
\begin{eqnarray*}
\begin{array}{rcl}
(x_{\overline{j-1},0}^{+})^{(k)} \cdot \left( \ffbox{j}_{a_1} \otimes \cdots \otimes \ffbox{j}_{a_k} \right) &=& \ffbox{j-1}_{a_1} \otimes \cdots \otimes \ffbox{j-1}_{a_k},\\
(x_{\overline{j},0}^{-})^{(k)} \cdot \left( \ffbox{j}_{a_1} \otimes \cdots \otimes \ffbox{j}_{a_k} \right) &=& \ffbox{j+1}_{a_1} \otimes \cdots \otimes \ffbox{j+1}_{a_k},\\
x_{i,0}^{\pm} \cdot \left( \ffbox{j}_{a_1} \otimes \cdots \otimes \ffbox{j}_{a_k} \right) &=& 0 \text{ otherwise.}
\end{array}
\end{eqnarray*}
The extremality of $\ffbox{1}_{a_1} \otimes \cdots \otimes \ffbox{1}_{a_k}$ follows.
\end{proof}

\subsection{The non-generic case}

Consider the tensor product of vector representations
$$V(\ffbox{1}_a) \otimes V(\ffbox{1}_{aq^{-2}}) \otimes \cdots  \otimes V(\ffbox{1}_{aq^{-2(k-1)}}).$$
Denote by $\tilde{V} = \tilde{V} \left(\begin{array}{c} \ffbox{1}_{\ } \\ \vdots_{\ } \\ \ffbox{k}_{a} \end{array} \right)$ its subvector space spanned by elements $$\ffbox{i_1}_a \otimes \ffbox{i_2}_{aq^{-2}} \otimes \cdots \otimes \ffbox{i_k}_{aq^{-2(k-1)}} \text{ with } i_1 < i_2 < \cdots < i_k$$
and $V(\ffbox{1} \cdots \ffbox{1}_a)$ its subvector space spanned by elements
$$\ffbox{j_1}_a \otimes \ffbox{j_2}_{aq^{-2}} \otimes \cdots \otimes \ffbox{j_k}_{aq^{-2(k-1)}} \text{ with } j_1 \geq j_2 \geq \cdots \geq j_k.$$

\subsubsection{Existence of the action}

\begin{prop}\label{propactex}
\begin{enumerate}
\item[(i)] The coproduct $\Delta_D$ endows $V(\ffbox{1}_a) \otimes \cdots \otimes V(\ffbox{1}_{aq^{-2(k-1)}})$ with a structure of $\U_q(sl_{n+1}^{tor})$-module if and only if 
\begin{align}
(n \text{ is even and } k \leq n+1) \text{ or } \left(n \text{ is odd and } k \leq \frac{n+1}{2} \right).
\tag{E}
\end{align}
\item[(ii)] The coproduct $\Delta_D$ endows $\tilde{V} \left(\begin{array}{c} \ffbox{1}_{\ } \\ \vdots_{\ } \\ \ffbox{k}_{a} \end{array} \right)$ with a structure of $\U_q(sl_{n+1}^{tor})$-module if and only if $k$ satisfies (E)\footnote{These representations were considered in \cite{feigin_representations_2013} when $d \neq 1$.}.
\item[(iii)] The coproduct $\Delta_D$ endows $V(\ffbox{1} \cdots \ffbox{1}_a)$ with a structure of $\U_q(sl_{n+1}^{tor})$-module for all $k \in \NN^{\ast}$.
\end{enumerate}
\end{prop}

\begin{proof}
The easy following lemma shows that (E) is needed in (i) and (ii).

\begin{lem}
\begin{itemize}
\item[(i)] Assume that $k > n+1$. The action of $x_{1,0}^{-}$ is not defined on the vector 
\begin{center}
\begin{tiny}
$\ffbox{1}_a \otimes $ \nolinebreak \renewcommand{\ffbox}[1]{\setbox9=\hbox{$\scriptstyle\overline{1}$}
\framebox[40pt][c]{\rule{0mm}{\ht9}${\scriptstyle #1}$}
} $ \ffbox{2+n+1}_{aq^{-2}} \otimes \cdots \otimes \ffbox{1+2(n+1)}_{aq^{-2(n+1)}} \otimes \ffbox{(n+1)^{n+3}}_{aq^{-2(n+2)}} \otimes \cdots \otimes \ffbox{(n+1)^{k}}_{aq^{-2(k-1)}}.$
\end{tiny}
\end{center}
\renewcommand{\ffbox}[1]{
\setbox9=\hbox{$\scriptstyle\overline{1}$}
\framebox[20pt][c]{\rule{0mm}{\ht9}${\scriptstyle #1}$}
}
\item[(ii)] Assume that  $n$ is odd and $k > \frac{n+1}{2} = l$. The action of $x_{1,0}^{-}$ is not defined on 

\begin{center}
\begin{tiny}
$\ffbox{1}_a \otimes \ffbox{3}_{aq^{-2}} \otimes \cdots \otimes \ffbox{1+l}_{aq^{-2(l-1)}} \otimes \ffbox{1+2l}_{aq^{-2l}}$ \nolinebreak \renewcommand{\ffbox}[1]{
\setbox9=\hbox{$\scriptstyle\overline{1}$}
\framebox[40pt][c]{\rule{0mm}{\ht9}${\scriptstyle #1}$}
} $ \otimes \ffbox{(n+1)^{l+2}}_{aq^{-2(l+1)}} \otimes \cdots \otimes \ffbox{(n+1)^{k}}_{aq^{-2(k-1)}}.$
\end{tiny}
\end{center}

\renewcommand{\ffbox}[1]{
\setbox9=\hbox{$\scriptstyle\overline{1}$}
\framebox[20pt][c]{\rule{0mm}{\ht9}${\scriptstyle #1}$}
}
\end{itemize}
\end{lem}

Now assume that $k$ satisfies (E) and let us show $(i)$ and $(ii)$. By Lemma \ref{lemsubm2} and Proposition \ref{proptpdef}, the action is well-defined on $V = V(\ffbox{1}_a) \otimes \cdots \otimes V(\ffbox{1}_{aq^{-2(k-1)}})$. Set $v = \ffbox{i_1} \otimes \ffbox{i_2} \otimes \cdots \ffbox{i_k}$ with $i_1 < i_2 < \cdots < i_k$. Let us show that $x_i^-(z) \cdot v \in \tilde{V}$ for any $i \in I$. From (\ref{formref}), this holds except perhaps if there exists $1 \leq u \leq k-1$ such that $\overline{i_u} = i$ and $\overline{i_{u+1}} = i+1$. But in this case $A_u = 0$. The proof is the same for the operators $x_i^+(z)$.

Fix $k \in \NN^{\ast}$ and let us prove (iii). The action of $x_i^-(z)$ is well-defined on all vectors \linebreak $\ffbox{j_1} \otimes \cdots \otimes \ffbox{j_k} \in V(\ffbox{1} \cdots \ffbox{1}_a)$: indeed in (\ref{formref}) the coefficients $A_u$ become
\begin{equation}\label{formau}
\displaystyle \prod_{\begin{tiny} \begin{array}{c} u < s \leq k\\ \overline{j_s} = i \end{array} \end{tiny}} \psi \left( q^{2(u-s) + j_s-j_u} \right) \prod_{\begin{tiny} \begin{array}{c} u < s \leq k\\ \overline{j_s} = i+1 \end{array} \end{tiny}} \psi \left(q^{2(u-s) + j_s-j_u+1} \right)^{-1}
\end{equation}
which are well-defined because $j_u \geq j_s$. Fix $v = \ffbox{i_1} \otimes \cdots \otimes \ffbox{i_k} \in V$, $v \notin V(\ffbox{1} \cdots \ffbox{1}_a)$. Let us show that the coefficient of $x_i^-(z) \cdot v$ computed on each vector in $V(\ffbox{1} \cdots \ffbox{1}_a)$ is zero: this is trivially the case except perhaps if there exists $1 \leq u \leq k-1$ such that $i_1 \geq \cdots \geq i_u+1 = i_{u+1} \geq \cdots \geq i_k$ and $\overline{i_u}=i$. But from (\ref{formau}) $A_u$ is equal to $0$ in this case. The result follows by Lemma \ref{lemquo} (the proof being the same for the $x_i^+(z)$).
\end{proof}

\begin{lem}
The $\U_q(sl_{n+1}^{tor})$-modules hence obtained are thin and integrable. Furthermore the action of the vertical quantum affine subalgebras are locally finite-dimensional.
\end{lem}

\begin{proof}
It is easy to show that the map
$$\ffbox{i_1}_a \otimes \cdots \otimes \ffbox{i_k}_{aq^{-2(k-1)}} \mapsto \prod_{1 \leq j \leq k} Y_{\overline{i_j},aq^{i_j-1-2(j-1)}}Y_{\overline{i_j-1},aq^{i_j-2(j-1)}}^{-1}$$
which associates to a vector $\ffbox{i_1}_a \otimes \cdots \ffbox{i_k}_{aq^{-2(k-1)}}$ its $\ell$-weight is bijective when $k$ satisfies (E) or when we assume in addition that $i_1 \geq \cdots \geq i_k$. Lemma \ref{tpalgv} gives the remaining statements.
\end{proof}

\subsubsection{Study of the tensor product module}

\begin{thm}\label{thmexng}
Assume that $k$ satisfies (E). Then the $\U_q(sl_{n+1}^{tor})$-module \linebreak $V = V(\ffbox{1}_a) \otimes \cdots \otimes V(\ffbox{1}_{aq^{-2(k-1)}})$ is an extremal loop weight module of $\ell$-weight $$e^{k \varpi_1}Y_{1,a}Y_{1,aq^{-2}} \cdots Y_{1, aq^{-2(k-1)}}Y_{0,aq}^{-1}Y_{0,aq^{-1}}^{-1} \cdots Y_{0, aq^{-2(k-1)+1}}^{-1}.$$
\end{thm}

\begin{proof}
The fact that the vector $$v = \ffbox{1}_a \otimes \ffbox{1}_{aq^{-2}} \otimes \cdots \otimes \ffbox{1}_{aq^{-2(k-1)}}$$ 
is extremal follows by Lemma \ref{extvectgen} (by setting $a_i = aq^{-2(i-1)}$).

Let us prove that $V$ is cyclic, i.e. for all $\ffbox{j_1} \otimes \cdots \otimes \ffbox{j_k} \in V$ there is $x \in \U_q(sl_{n+1}^{tor})$ sending $v$ on it. As $v$ is extremal, it suffices to show that when $j_s \geq 1$ for all $1 \leq s \leq k$. This is an easy consequence of

\begin{lem}\label{lemrefirred}
For all $1 \leq s \leq k$ and $t \geq 1$, there exists $y \in \U_q(sl_{n+1}^{tor})$ such that
$$y \cdot \left(\ffbox{j_1} \otimes \cdots \ffbox{j_{s-1}} \otimes \ffbox{t} \otimes \ffbox{1} \otimes \cdots \otimes \ffbox{1} \right) = \ffbox{j_1} \otimes \cdots \ffbox{j_{s-1}} \otimes \ffbox{t+1} \otimes \ffbox{1} \otimes \cdots \otimes \ffbox{1}.$$
\end{lem}

\medskip

\noindent Its proof is similar to the one of Lemma \ref{lemirrmod}: by (\ref{formref}), the vector $$\ffbox{j_1} \otimes \cdots \ffbox{j_{s-1}} \otimes \ffbox{t+1} \otimes \ffbox{1} \otimes \cdots \otimes \ffbox{1}$$ occurs in $$x_{\overline{t}}^-(z) \cdot \left( \ffbox{j_1} \otimes \cdots \otimes \ffbox{j_{s-1}} \otimes \ffbox{t} \otimes \ffbox{1} \otimes \cdots \otimes \ffbox{1} \right)$$
with the coefficient $A_s \neq 0$. As $V$ is thin, we get the existence of the desired $y \in \U_q(sl_{n+1}^{tor})$.
\end{proof}

\subsubsection{Study of $V(\overline{\underline{\vert \ {\tiny 1} \ \vert}} \cdots \overline{\underline{\vert \ {\tiny 1} \ \vert}}_a)$}

\begin{prop}\label{propexmod}
The $\U_q(sl_{n+1}^{tor})$-module $V(\ffbox{1} \cdots \ffbox{1}_a)$ associated to $a \in \CC^{\ast}$ and $k \in \NN^{\ast}$ is an irreducible extremal loop weight module of $\ell$-weight $$e^{k \varpi_1}Y_{1,a}Y_{1,aq^{-2}} \cdots Y_{1, aq^{-2(k-1)}}Y_{0,aq}^{-1}Y_{0,aq^{-1}}^{-1} \cdots Y_{0, aq^{-2(k-1)+1}}^{-1}.$$
\end{prop}

\begin{proof}
The fact that
$$\ffbox{1}_a \otimes \ffbox{1}_{aq^{-2}} \otimes \cdots \otimes \ffbox{1}_{aq^{-2(k-1)}}$$ 
is extremal still follows by Lemma \ref{extvectgen}.

Assume in (\ref{formref}) that $j_1 \geq j_2 \geq \cdots \geq j_k$. From (\ref{formau}) the coefficients $A_u$ occurring in it are non zero for each $u$. Then by Lemma \ref{lemirrmod} $V(\ffbox{1} \cdots \ffbox{1}_a)$ is irreducible.
\end{proof}

We recover the extremal loop weight $\U_q(sl_{4}^{tor})$-module constructed in \cite[Section 5]{mansuy_quantum_2012}.

\begin{prop}\label{proprecrep}
The $\U_q(sl_{4}^{tor})$-modules $V(e^{2\varpi_1}Y_{1,a}Y_{1,aq^{-2}}Y_{0,aq}^{-1}Y_{0,aq^{-1}}^{-1})$ and \linebreak $V(\ffbox{1} \ffbox{1}_a)$ are isomorphic.
\end{prop}

\begin{proof}
The action of $\U_q(sl_{4}^{tor})$ on $V(e^{2\varpi_1}Y_{1,a}Y_{1,aq^{-2}}Y_{0,aq}^{-1}Y_{0,aq^{-1}}^{-1})$ is explicitly known (see \cite{mansuy_quantum_2012}) and is the same than on $V(\ffbox{1} \ffbox{1}_a)$.
\end{proof}


\subsubsection{Study of $\tilde{V}$}

\begin{prop}
Assume that
\begin{align}
(n \text{ is even and } k < n+1) \text{ or } \left(n \text{ is odd and } k < \frac{n+1}{2} \right).
\tag{E'}
\end{align}
Then $\tilde{V} = \tilde{V} \left(\begin{array}{c} \ffbox{1}_{\ } \\ \vdots_{\ } \\ \ffbox{k}_a \end{array} \right)$ is an irreducible $\U_q(sl_{n+1}^{tor})$-module.
\end{prop}

\begin{proof}
Assume in (\ref{formref}) that $j_1 < j_2 < \cdots < j_k$. One checks easily that if $(E')$ is satisfied, the coefficients $A_u$ are equal to $0$ if and only if $j_{u+1} = j_u + 1$. Then we prove that $\tilde{V}$ is irreducible as in Lemma \ref{lemirrmod}.
\end{proof}

We would like to construct extremal loop weight modules of $\ell$-weight $e^{\varpi_k}Y_{k,a}Y_{0,aq^{d_k}}^{-1}$ for $1 \leq k \leq n$. From $\tilde{V}$ it is natural to consider the vector
$$v=\ffbox{1}_a \otimes \ffbox{2}_{aq^{-2}} \otimes \cdots \otimes \ffbox{k}_{aq^{-2(k-1)}} \in \tilde{V}.$$
But when $k$ satisfies $(E')$, one shows by straightforward computations that $v$ is not extremal for the horizontal quantum affine subalgebra.
We study the case $n=2r+1$ odd and $k = r+1$ in the rest of this section.

 Denote by $V \left(\begin{array}{c} \ffbox{1}_{\ } \\ \vdots_{\ } \\ \ffbox{r+1}_a \end{array} \right)$ the subvector space of $\tilde{V} \left(\begin{array}{c} \ffbox{1}_{\ } \\ \vdots_{\ } \\ \ffbox{r+1}_a \end{array} \right)$ generated by
$$\ffbox{i_1}_a \otimes \ffbox{i_2}_{aq^{-2}} \otimes \cdots \otimes \ffbox{i_{r+1}}_{aq^{-2r}} \text{ with } i_1 < i_2 < \cdots < i_{r+1}$$
and $i_{r+1}-i_1 \leq n$.

\begin{prop}
The coproduct $\Delta_D$ endows $V = V \left(\begin{array}{c} \ffbox{1}_{\ } \\ \vdots_{\ } \\ \ffbox{r+1}_a \end{array} \right)$ with a structure of irreducible and thin $\U_q(sl_{n+1}^{tor})$-module.
\end{prop}

\begin{proof}
Let $i_1=i < i_2 < \cdots < i_r < i_{r+1} = i+1+(n+1)$ and $v$ the corresponding vector in $\tilde{V}$. From (\ref{formau}) the coefficient $A_1$ in $x_{\overline{i}}^{-}(z) \cdot v$ is equal to $0$. Then by Lemma \ref{lemquo} the subvector space $V$ can be endowed with a structure of $\U_q(sl_{n+1}^{tor})$-module (the proof is the same for the operators $x_{\overline{i}}^+(z)$). Furthermore as $\tilde{V}$ is thin, $V$ is thin also.

We prove that $V$ is irreducible as in Lemma \ref{lemirrmod}: indeed if $j_1 < \cdots < j_{r+1}$ and $j_{r+1}-j_1 \leq n$, the coefficients $A_u$ in (\ref{formref}) are equal to zero only if $j_{u+1}=j_u+1$.
\end{proof}

Fix a sequence $(i_1 < i_2 < \cdots < i_{r+1})$ with $i_{r+1}-i_1 \leq n$. Let $k \in \ZZ$ and $1\leq t \leq r+1$ defined such that
$$(k-1)(n+1) \leq i_1 < \cdots < i_t \leq k(n+1) < i_{t+1} < \cdots < i_{r+1}.$$
We associate to this sequence a couple $(s,T)$ of an integer $s = -t+k(r+1)$ and a Young tableaux $T$ of shape ($r+1$)
$$(i_{t+1} -k(n+1) < \cdots < i_{r+1}-k(n+1) < i_1 - (k-1)(n+1) < \cdots < i_t - (k-1)(n+1)).$$
This correspondence is bijective. Then we set
$$T_{aq^{2s}} = \ffbox{i_1}_a \otimes \ffbox{i_2}_{aq^{-2}} \otimes \cdots \otimes \ffbox{i_{r+1}}_{aq^{-2r}} \in V \left(\begin{array}{c} \ffbox{1}_{\ } \\ \vdots_{\ } \\ \ffbox{r+1}_a \end{array} \right),$$
where $(i_1 < \cdots < i_{r+1})$ is the sequence associated to $(s,T)$. The family $\{T_{aq^{2s}}\}_{T,s}$ is a basis of $V \left(\begin{array}{c} \ffbox{1}_{\ } \\ \vdots_{\ } \\ \ffbox{r+1}_a \end{array} \right)$.

\begin{prop}\label{propacttaby}
The action of $\U_{q}(sl_{n+1}^{tor})$ on $V \left(\begin{array}{c} \ffbox{1}_{\ } \\ \vdots_{\ } \\ \ffbox{r+1}_a \end{array} \right)$ is given by
\begin{eqnarray}\label{actonpm}
\begin{array}{ccl}
x_{i}^+(z) \cdot T_b &=& \displaystyle \sum_{\begin{tiny}
\begin{array}{c}
1 \leq p \leq r+1\\
i_p = i+1
\end{array}
\end{tiny}} \delta(bq^{i-2p}z) (\tilde{e}_i \cdot T)_{bq^{-2 \delta_{i,0}}},\\
 & & \\
x_{i}^-(z) \cdot T_b &=& \displaystyle \sum_{\begin{tiny}
\begin{array}{c}
1 \leq p \leq r+1\\
i_p = i
\end{array}
\end{tiny}} \delta(bq^{i-2p}z) (\tilde{f}_i \cdot T)_{bq^{2 \delta_{i,0}}},
\end{array}
\end{eqnarray}
\begin{eqnarray}\label{actoncart}
\phi_{i}^\pm(z) \cdot T_b &=& \left( \displaystyle \prod_{{\tiny \begin{array}{c} p \\ i_p = i+1 \\ i_{p-1} \neq i \end{array}}} \psi(bq^{i -2p+2}z)^{-1} \prod_{{\tiny \begin{array}{c} p \\ i_p = i \\ i_{p+1} \neq i+1 \end{array}}} \psi(bq^{i-2p}z) \right) \cdot T_b
\end{eqnarray}
where $T = (1\leq i_1 < i_2 < \cdots < i_{r+1} \leq n+1)$ is a Young tableau of shape ($r+1$) and $i \in I, b \in aq^{2 \ZZ}$.
\end{prop}

The operators $\tilde{e}_i$ and $\tilde{f}_i$ in (\ref{actonpm}) are the Kashiwara operators in the crystal of semi-standard Young tableaux of shape $(r+1)$. Their action is such that: for $i\neq 0$ we have $\tilde{e}_i \cdot T = T'$ where $T'$ is obtained from $T$ by replacing $i+1$ by $i$. If it is not possible (i.e. when we have both $i+1$ and $i$ in $T$ or when $i+1$ does not occur in $T$), then it is zero. Similarly $\tilde{f}_i \cdot T= T''$ or $0$, where $T''$ is given by replacing $i$ by $i+1$. For the action of $\tilde{e}_0$ and $\tilde{f}_0$, we have
\begin{equation*}
\begin{split}
\tilde{e}_0 \cdot T &=
        \begin{cases}
        0 &\text{ if $i_1\neq 1$ or $i_{r+1}=n+1$},
                \\  (i_2,\cdots,i_{r+1},n+1) &\text{ if $i_1=1$ and $i_{r+1}\neq n+1$,}
  \end{cases}
\\
\tilde{f}_0 \cdot T &=
        \begin{cases}
        0 &\text{ if $i_1 = 1$ or $i_{r+1} \neq n+1$},
                \\  (1,i_1,\cdots,i_{r})&\text{ if $i_1 \neq 1$ and $i_{r+1} = n+1$.}
  \end{cases} 
\end{split}
\end{equation*}

\begin{proof}
This is a rewriting of (\ref{formref}) for this subquotient module.
\end{proof}

\begin{thm}\label{tpeflwm}
The representation $V \left(\begin{array}{c} \ffbox{1}_{\ } \\ \vdots_{\ } \\ \ffbox{r+1}_a \end{array} \right)$ is isomorphic to the extremal fundamental loop weight module $V(e^{\varpi_{r+1}}Y_{r+1, aq^{-r-2}}Y_{0, aq^{-1}}^{-1})$.

In particular, $V \left(\begin{array}{c} \ffbox{1}_{\ } \\ \vdots_{\ } \\ \ffbox{r+1}_a \end{array} \right)$ is $\ell$-extremal of $\ell$-weight $e^{\varpi_{r+1}}Y_{r+1, aq^{-r-2}}Y_{0, aq^{-1}}^{-1}$.
\end{thm}

\begin{proof}
We got formulas of the action of $\U_q(sl_{n+1}^{tor})$ on $V(e^{\varpi_{r+1}}Y_{r+1, aq^{-r-2}}Y_{0, aq^{-1}}^{-1})$ in \cite[Theorem 4.12]{mansuy_quantum_2012} in terms of monomial crystals. They can be rewritten in Young tableaux notations by setting for all $b \in \CC^{\ast}$ and $j \in \ZZ$
$$\ffbox{j}_b = Y_{j-1, bq^{j}}^{-1}Y_{j, bq^{j-1}}.$$
We recover in this way formulas (\ref{actonpm}) and (\ref{actoncart}).
\end{proof}

\subsection{Finite-dimensional representations at roots of unity} Set $L \geq 1$ and let $\epsilon$ be a primitive $[(n+1)L]$-root of unity. Recall that $\U_\epsilon(sl_{n+1}^{tor})'$ is the algebra defined as $\U_q(sl_{n+1}^{tor})$ with $\epsilon$ instead of $q$ (without divided power and derivation element).

For $a \in \CC^{\ast}$, let $V(\ffbox{1}_a)_\epsilon$ be the $\U_\epsilon(sl_{n+1}^{tor})'$-module obtained from $V(\ffbox{1}_a)$ by specializing $q$ at $\epsilon$ in (\ref{actrv}) and by taking a quotient (see \cite{mansuy_quantum_2012}): it is irreducible of dimension $[(n+1)L]$.

\begin{thm}\label{thmrootunit1}
Let $k \in \NN^{\ast}$ and $a_1, \cdots , a_k \in \CC^{\ast}$ be such that $$\frac{a_i}{a_j} \notin \{1, \epsilon^{n+1}, \cdots , \epsilon^{(n+1)(L-1)}\} \quad \text{ and } \quad \frac{a_i}{a_j} \notin \{\epsilon^{\pm 2}, \epsilon^{\pm 2 +(n+1)}, \cdots , \epsilon^{\pm 2 + (n+1)(L-1)} \}$$ for all $1 \leq i,j \leq k$. Then $\Delta_D$ endows the tensor product
$$V(\ffbox{1}_{a_1})_\epsilon \otimes V(\ffbox{1}_{a_2})_\epsilon \otimes \cdots \otimes V(\ffbox{1}_{a_k})_\epsilon$$
with a well-defined structure of $\U_\epsilon(sl_{n+1}^{tor})'$-module. Furthermore it is irreducible and finite-dimensional of dimension $[(n+1)L]^{k}$.
\end{thm}

\begin{proof}
The proof is the same than when $q$ is not a root of unity.
\end{proof}

\section{Tensor products of $\U_{q}(sl_{2r+2}^{tor})$-modules $V(Y_{r+1,a}Y_{0,aq^{r+1}}^{-1})$}

In the whole section we assume that $n=2r+1$ is odd with $r \geq 1$. We study tensor products of representations $V \left(\begin{array}{c} \ffbox{1}_{\ } \\ \vdots_{\ } \\ \ffbox{r+1}_a \end{array} \right)$. We obtain new extremal loop weight modules for $\U_q(sl_{n+1}^{tor})$.

In the first part we study a tensor product with two factors (Proposition \ref{proptpdef2}). The case of tensor products with generic complex parameters is considered in the second part (Theorem \ref{thmexgen2}). In the last part we get new irreducible finite-dimensional modules over $\U_\epsilon(sl_{n+1}^{tor})'$ by specializing $q$ at a root of unity $\epsilon$.

\subsection{Existence conditions}

\begin{prop}\label{proptpdef2}
The action of $\U_q(sl_{n+1}^{tor})$ on the tensor product $$V \left(\begin{array}{c} \ffbox{1}_{\ } \\ \vdots_{\ } \\ \ffbox{r+1}_a \end{array} \right) \otimes V \left(\begin{array}{c} \ffbox{1}_{\ } \\ \vdots_{\ } \\ \ffbox{r+1}_b \end{array} \right)$$
is well-defined if and only if $\dfrac{a}{b} \notin q^{2 \ZZ}$.
\end{prop}

\begin{proof}
Let us consider $$T^{(1)} = (1 \leq i_1^{(1)} < \cdots < i_\ell^{(1)} \leq n+1) \text{ and } T^{(2)} = (1 \leq i_1^{(2)} < \cdots < i_\ell^{(2)} \leq n+1)$$ two Young tableaux of shape $(r+1)$ and $\alpha \in aq^{2\ZZ}, \beta \in bq^{2\ZZ}$. Consider the vector $x_i^-(z) \cdot \left(T_\alpha^{(1)} \otimes T_\beta^{(2)} \right)$. It is a linear combination of $T_\alpha^{(1)} \otimes(\tilde{f}_i \cdot T^{(2)})_{\beta q^{2 \delta_{i,0}}}$ and $(\tilde{f}_i \cdot T^{(1)})_{\alpha q^{2\delta_{i,0}}} \otimes T_\beta^{(2)}$. The coefficient on the first vector is well-defined. For the second one, it is equal to
$$\sum_{\begin{tiny}
\begin{array}{c}
1 \leq u \leq r+1\\
i_u^{(1)} = i
\end{array}
\end{tiny}} \delta(\alpha q^{i-2u}z) \left( \displaystyle \prod_{{\tiny \begin{array}{c} v \\ i_v^{(2)} = i+1 \\ i_{v-1}^{(2)} \neq i \end{array}}} \psi \left(\frac{\beta}{\alpha} q^{2(u-v)+2} \right)^{-1} \prod_{{\tiny \begin{array}{c} v \\ i_v^{(2)} = i \\ i_{v+1}^{(2)} \neq i+1 \end{array}}} \psi \left(\frac{\beta}{\alpha} q^{2(u-v)} \right) \right)$$
which is well-defined if and only if $\dfrac{\alpha}{\beta} q^{2(v-u)} \neq 1$. Then a necessary and sufficient condition is that $\dfrac{a}{b} \notin q^{2\ZZ}$. Similar computations for $x_i^+(z)$ lead to the same condition. Lemma \ref{lemsubm} gives the result.
\end{proof}

\subsection{The generic case}

\begin{thm}\label{thmexgen2}
Let $k \in \NN^{\ast}$ and $a_1, \cdots, a_k \in \CC^{\ast}$ be such that $\frac{a_i}{a_j} \notin q^{2\ZZ}$ for each $i < j$. Then the tensor product
 $$V = V \left(\begin{array}{c} \ffbox{1}_{\ \ } \\ \vdots_{\ \ } \\ \ffbox{r+1}_{a_1} \end{array} \right) \otimes V \left(\begin{array}{c} \ffbox{1}_{\ \ } \\ \vdots_{\ \ } \\ \ffbox{r+1}_{a_2} \end{array} \right) \otimes \cdots \otimes V \left(\begin{array}{c} \ffbox{1}_{\ \ } \\ \vdots_{\ \ } \\ \ffbox{r+1}_{a_k} \end{array} \right)$$
is an irreducible and thin extremal loop weight $\U_q(sl_{n+1}^{tor})$-module of $\ell$-weight $$e^{k \varpi_{r+1}}Y_{r+1,a_1}\cdots Y_{r+1, a_k}Y_{0,a_1 q^{r+1}}^{-1} \cdots Y_{0, a_k q^{r+1}}^{-1}.$$
\end{thm}

\begin{proof}
The existence of the action follows by Proposition \ref{proptpdef2} and Lemma \ref{lemsubm2}. The representation $V$ hence obtained is a subquotient of tensor products of vector representations and we are led to similar studies as the above section. Then we prove in the same way the remaining statements of Theorem \ref{thmexgen2}. We do not detail these points.

%
%
%
\end{proof}

\subsection{Finite-dimensional representations at roots of unity}

Set $L > 1$ and let $\epsilon$ be a primitive $[2L]$-root of unity. For all $a \in \CC^{\ast}$, let $V \left(\begin{array}{c} \ffbox{1}_{\ } \\ \vdots_{\ } \\ \ffbox{r+1}_a \end{array} \right)_\epsilon$ be the $\U_\epsilon(sl_{n+1}^{tor})'$-module obtained from $V \left(\begin{array}{c} \ffbox{1}_{\ } \\ \vdots_{\ } \\ \ffbox{r+1}_a \end{array} \right)$ by specializing $q$ at $\epsilon$ in (\ref{actonpm}) and (\ref{actoncart}) and by taking a quotient (see \cite{mansuy_quantum_2012}): it is an irreducible and finite-dimensional $\U_\epsilon(sl_{n+1}^{tor})'$-module of dimension $L \cdot \left( \begin{array}{c} n+1 \\ r+1 \end{array} \right)$. We prove as in the generic case that

\begin{thm}\label{thmrootunit2}
Let $k \in \NN^{\ast}$ and $a_1, \cdots, a_k \in \CC^{\ast}$ be such that $\frac{a_i}{a_j} \notin \epsilon^{2\ZZ}$ for all $i < j$. Then $\Delta_D$ endows the tensor product
$$V \left(\begin{array}{c} \ffbox{1}_{\ \ } \\ \vdots_{\ \ } \\ \ffbox{r+1}_{a_1} \end{array} \right)_\epsilon \otimes V \left(\begin{array}{c} \ffbox{1}_{\ \ } \\ \vdots_{\ \ } \\ \ffbox{r+1}_{a_2} \end{array} \right)_\epsilon \otimes \cdots \otimes V \left(\begin{array}{c} \ffbox{1}_{\ \ } \\ \vdots_{\ \ } \\ \ffbox{r+1}_{a_k} \end{array} \right)_\epsilon$$
with a well-defined structure of $\U_\epsilon(sl_{n+1}^{tor})'$-module. Furthermore it is irreducible and finite-dimensional of dimension $\left[ L \cdot \left( \begin{array}{c} n+1 \\ r+1 \end{array} \right) \right]^{k}$.
\end{thm}

\section{Further possible developments}

In a further paper we expect to construct extremal loop weight modules for quantum toroidal algebras of any type. We give here an example of such construction for the quantum toroidal algebra $\U_q(\Glie^{tor})$ of type $D_4$.

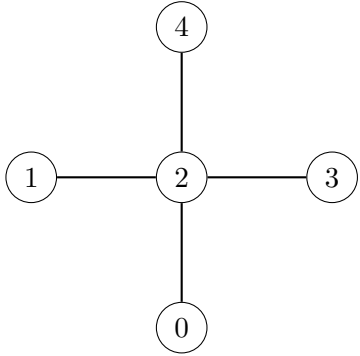
\newcommand{\eq}[1][r]
   {\ar@<-0pt>@{-}[#1]}

\begin{center} \begin{tikzpicture}
\tikzstyle{point}=[circle,draw]
\tikzstyle{ligne}=[thick]
\tikzstyle{pointille}=[thick,dotted]

\node (1) at (-2, 0) [point]{1};
\node (2) at (0,0) [point] {2};
\node (3) at (2,0) [point] {3};
\node (4) at (0,2) [point] {4};
\node (5) at (0,-2) [point] {0};

\draw [ligne] (1) -- (2);
\draw [ligne] (2) -- (3);
\draw [ligne] (2) -- (4);
\draw [ligne] (2) -- (5);

\end{tikzpicture}\end{center}

Let us consider the tensor product $V(e^{\Lambda_1}Y_{1,a}) \otimes V(e^{-\Lambda_0}Y_{0,aq^{2}}^{-1})$.

\begin{prop}
The coproduct $\Delta_D$ endows $V(e^{\Lambda_1}Y_{1,a}) \otimes V(e^{-\Lambda_0}Y_{0,aq^{2}}^{-1})$ with a structure of $\U_q(\Glie^{tor})$-module.
\end{prop}

\begin{proof}
This follows by Theorem \ref{tpexist}, the proof holding in type $D_4$.
\end{proof}

Let $T(e^{\varpi_1}Y_{1,a}Y_{0,aq^{2}}^{-1})$ be the sub-$\U_q(\Glie^{tor})$-module of $V(e^{\Lambda_1}Y_{1,a}) \otimes V(e^{-\Lambda_0}Y_{0,aq^{2}}^{-1})$ generated by $v^+ \otimes v^-$ where $v^+$ (resp. $v^-$) is a $\ell$-highest weight vector of $V(e^{\Lambda_1}Y_{1,a})$ (resp. a $\ell$-lowest weight vector of $V(e^{-\Lambda_0}Y_{0,aq^{2}})$).

\begin{thm}\label{extypd}
$T(e^{\varpi_1}Y_{1,a}Y_{0,aq^{2}}^{-1})$ is an extremal loop weight  $\U_q(\Glie^{tor})$-module of $\ell$-weight $e^{\varpi_1}Y_{1,a}Y_{0,aq^{2}}^{-1}$.
\end{thm}

\begin{proof}

One can compute the action of $\U_q(\Glie^{tor})$ on $T(e^{\varpi_1}Y_{1,a}Y_{0,aq^{2}}^{-1})$ as in the proof of Proposition \ref{repvect}: it is given by the monomial realization $\mathcal{M}(e^{\varpi_1}Y_{1,1}Y_{0,q^{2}}^{-1})$ of $\mathcal{B}(\varpi_1)$ and formulas in \cite[Theorem 4.12]{mansuy_quantum_2012}. We do not detail these calculations. The result follows.
\end{proof}

As in type $A$, $T(e^{\varpi_1}Y_{1,a}Y_{0,aq^{2}}^{-1})$ is irreducible and $\mathrm{Res}(T(e^{\varpi_1}Y_{1,a}Y_{0,aq^{2}}^{-1}))$ is isomorphic to the extremal fundamental weight $\U_q(\hat{\Glie})$-module $V(\varpi_1)$. Furthermore by specializing $q$ at roots of unity, we have

\begin{prop}
Set $L \geq 1$ and assume that $\epsilon$ is a primitive $[6L]$-root of unity. There exists an irreducible $\U_\epsilon(\Glie^{tor})'$-module of dimension $8L$.
\end{prop}

\begin{proof}
As in \cite{mansuy_quantum_2012} one checks that the $\U_\epsilon(\Glie^{tor})'$-module obtained from $T(e^{\varpi_1}Y_{1,a}Y_{0,aq^{2}}^{-1})$ by specializing $q$ at $\epsilon$ has an irreducible and finite-dimensional quotient.
\end{proof}

\section*{Appendix. Action on the $\U_q(sl_{n+1}^{tor})$-module $T(e^{\varpi_1}Y_{1,a}Y_{0,aq}^{-1})$ }

The Appendix is devoted to the proof of Proposition \ref{repvect}. We assume in this section that $a$ is equal to one. Then the result will follow by twisting the action by $t_{a}$. To simplify the notations, we denote by $Y_{i,l}$ the variable $Y_{i,q^{l}}$ for $i \in I$ and $l\in \ZZ$.

So let us consider the $\U_q(sl_{n+1}^{tor})$-module
$$T(e^{\varpi_1}Y_{1,0}Y_{0,1}^{-1}) = \U_q(sl_{n+1}^{tor}) \cdot \left( v_{Y_{1,0}} \otimes w_{Y_{0,1}^{-1}} \right) \subseteq V(e^{\Lambda_1}Y_{1,0}) \otimes V(e^{-\Lambda_0}Y_{0,1}^{-1})$$
where $v_{Y_{1,0}}$ (resp. $w_{Y_{0,1}^{-1}}$) is a $\ell$-highest weight vector of $V(e^{\Lambda_1}Y_{1,0})$ (resp. a $\ell$-lowest weight vector of $V(e^{-\Lambda_0}Y_{0,1}^{-1})$).

To study the action of $\U_q(sl_{n+1}^{tor})$ on $T(e^{\varpi_1}Y_{1,0}Y_{0,1}^{-1})$, we determine the action in each $\U_q(\hat{sl}_{2})'$-direction on the met vectors. In particular, we will have to study $\U_q(\hat{sl}_{2})'$-modules generated by $\ell$-highest weight vectors of $\ell$-highest weight of the form $Y_{l}$ or $Y_{l}Y_{m}$ ($l, m \in \ZZ, l \neq m$). By the theory of Weyl modules \cite{chari_integrable_2001}, these representations are known: they are isomorphic to the simple $\ell$-highest weight $\U_q(\hat{sl}_2)'$-modules of $\ell$-highest weights $Y_{l}$ and $Y_{l}Y_{m}$ respectively. This will be used in the whole computations.

Consider the action on the vector $v_{Y_{1,0}} \otimes w_{Y_{0,1}^{-1}}$: we have
\begin{eqnarray*}
x_{0}^{+}(z) \cdot \left( v_{Y_{1,0}} \otimes w_{Y_{0,1}^{-1}} \right) &=& \delta(z) \cdot v_{Y_{1,0}} \otimes w_{Y_{1,0}^{-1}Y_{n,0}^{-1}Y_{0,-1}},\\
x_{1}^{-}(z) \cdot \left( v_{Y_{1,0}} \otimes w_{Y_{0,1}^{-1}} \right) &=& \delta(qz) \cdot v_{Y_{0,1}Y_{1,2}^{-1}Y_{2,1}} \otimes w_{Y_{0,1}^{-1}},\\
x_{i}^{\pm}(z) \cdot \left( v_{Y_{1,0}} \otimes w_{Y_{0,1}^{-1}} \right) &=& 0 \text{ in all the other cases,}
\end{eqnarray*}
where $v_{Y_{0,1}Y_{1,2}^{-1}Y_{2,1}} = x_{1,0}^{-} \cdot v_{Y_{1,0}}$ and $w_{Y_{1,0}^{-1}Y_{3,0}^{-1}Y_{0,-1}} = x_{0,0}^{+} \cdot w_{Y_{0,1}^{-1}}$ are vectors of $\ell$-weight $Y_{0,1}Y_{1,2}^{-1}Y_{2,1}$ and $Y_{1,0}^{-1}Y_{3,0}^{-1}Y_{0,-1}$ respectively.\\


Let us show by induction that for all $j \geq 1$,
\begin{align}\label{forminter}
x_{j}^{+}(z) & \cdot \left( v_{Y_{0,1}Y_{j,j+1}^{-1}Y_{j+1,j}} \otimes w_{Y_{0,1}^{-1}} \right) = \delta(q^{j}z) v_{Y_{0,1}Y_{j-1,j}^{-1}Y_{j,j-1}} \otimes w_{Y_{0,1}^{-1}}, \nonumber \\ 
x_{j+1}^{-}(z) & \cdot \left( v_{Y_{0,1}Y_{j,j+1}^{-1}Y_{j+1,j}} \otimes w_{Y_{0,1}^{-1}} \right) = \delta(q^{j+1}z) v_{Y_{0,1}Y_{j+1,j+2}^{-1}Y_{j+2,j+1}} \otimes w_{Y_{0,1}^{-1}},\\ 
x_{j}^{\pm}(z) & \cdot \left( v_{Y_{0,1}Y_{j,j+1}^{-1}Y_{j+1,j}} \otimes w_{Y_{0,1}^{-1}} \right) = 0 \text{ in all the other cases,} \nonumber
\end{align}
where $v_{Y_{0,1}Y_{j+1,j+2}^{-1}Y_{j+2,j+1}}$ is a $\ell$-weight vector of weight $Y_{0,1}Y_{j+1,j+2}^{-1}Y_{j+2,j+1}$.

For $j \neq n, 0 \mod (n+1)$, formulas (\ref{forminter}) can be obtained by easy computations, by setting $v_{Y_{0,1}Y_{j+1,j+2}^{-1}Y_{j+2,j+1}} = x_{\overline{j+1},0}^{-} \cdot v_{Y_{0,1}Y_{j,j+1}^{-1}Y_{j+1,j}}$.

We improve what it is happening for the vectors
 $$v_{Y_{0,1}Y_{n, k(n+1)}^{-1} Y_{0,k(n+1)-1}} \text{ and } v_{Y_{0,1}Y_{0,k(n+1)+1}^{-1} Y_{1,k(n+1)}} \; \text{ with } k \geq 1.$$

The particular case $n=3$ and $k=1$ is easy to compute: the $\hat{\U}_0$-module \linebreak $\hat{\U}_0 \cdot v_{Y_{0,1} Y_{3,4}^{-1} Y_{0,3}}$ is isomorphic to the Kirillov-Reshetikhin $\U_q(\hat{sl}_2)'$-module of $\ell$-highest weight $Y_{0,3}Y_{0,1}$ on which the action is well-known. We set in this case $v_{Y_{0,1} Y_{0,5}^{-1}Y_{1,4}} = \psi(q^{-2})^{-1} x_{0,0}^{-} \cdot v_{Y_{0,1} Y_{3,4}^{-1} Y_{0,3}}$ and $v_{Y_{0,3}^{-1} Y_{0,5}^{-1}Y_{1,4}Y_{1,2}Y_{3,2}} = x_{0,0}^- \cdot v_{Y_{0,1} Y_{0,5}^{-1}Y_{1,4}}$ (of $\ell$-weight $Y_{0,1} Y_{0,5}^{-1}Y_{1,4}$ and $Y_{0,3}^{-1} Y_{0,5}^{-1}Y_{1,4}Y_{1,2}Y_{3,2}$ respectively) and we show that formulas (\ref{forminter}) hold for them.

Assume that $n \neq 3$ or $k \neq 1$ and set $m = Y_{0,1}Y_{n, k(n+1)}^{-1} Y_{0,k(n+1)-1}$. Then $\hat{\U}_0 \cdot v_{m}$ is isomorphic to the simple $\ell$-highest weight $\U_q(\hat{sl}_2)'$-module of $\ell$-highest weight $Y_{0,k(n+1)-1}Y_{0,1}$. The action on it is computed in \cite{hernandez_simple_2010, mansuy_quantum_2012}: let us denote 
\begin{align*}
m_1 = Y_{0,3}^{-1}Y_{1,2}Y_{n,2}Y_{n, k(n+1)}^{-1} Y_{0,k(n+1)-1} \quad \text{and} \quad m_2 = Y_{0,1} Y_{0,k(n+1)+1}^{-1}Y_{1,k(n+1)}.
\end{align*}
Then there is vectors $v_{m_1}$ and $v_{m_2}$ of $\ell$-weight $m_1$ and $m_2$ respectively such that
\begin{eqnarray*}
x_{0}^{-}(z) \cdot v_{m} &=& \psi(q^{k(n+1)-2}) \delta(q^{2}z) \psi(q^2z)^{-1} \cdot v_{m_1}\\
& & + \psi(q^{2-k(n+1)}) \delta(q^{k(n+1)}z) \psi(q^2z)^{-1} \cdot v_{m_2}.
\end{eqnarray*}

%

The action on $v_{m} \otimes w_{Y_{0,1}^{-1}}$ is:
\begin{eqnarray*}
x_{0}^{+}(z) \cdot \left( v_{m} \otimes w_{Y_{0,1}^{-1}} \right) &=& 0, \\
x_{0}^{-}(z) \cdot \left( v_{m} \otimes w_{Y_{0,1}^{-1}} \right) &=& \psi(q^{k(n+1)-2}) \delta(q^{2}z) \psi(1)^{-1} \cdot v_{m_1} \otimes w_{Y_{0,1}^{-1}}\\
& & + \psi(q^{2-k(n+1)}) \delta(q^{k(n+1)}z) \psi(q^{2-k(n+1)})^{-1} \cdot v_{m_2} \otimes w_{Y_{0,1}^{-1}}\\
&=& \delta(q^{k(n+1)}z) \cdot v_{m_2} \otimes w_{Y_{0,1}^{-1}},\\
x_{n}^{+}(z) \cdot \left( v_{m} \otimes w_{Y_{0,1}^{-1}} \right) &=& \delta(q^{k(n+1)-1}z) \cdot v_{Y_{0,1}Y_{n-1, k(n+1)-1}^{-1} Y_{n,k(n+1)-2}} \otimes w_{Y_{0,1}^{-1}},\\
x_{i}^{\pm}(z) \cdot \left( v_{m} \otimes w_{Y_{0,1}^{-1}} \right) &=& 0 \text{ in all the other cases.}
\end{eqnarray*}


Then we obtain formulas (\ref{forminter}) for all $j \geq 1$. By the same process we get similar formulas for vectors of the form $v_{Y_{1,0}} \otimes w_{Y_{1,0}^{-1}Y_{n+1-j,1-j}^{-1}Y_{n+2-j,-j}}$ with $j \geq 1$. Proposition \ref{repvect} follows by setting
$$\ffbox{j}_1 = v_{Y_{0,1}Y_{j,j+1}^{-1}Y_{j+1,j}} \otimes w_{Y_{0,1}^{-1}} \text{ for } j \geq 0$$
and
$$\ffbox{j}_1 = v_{Y_{1,0}} \otimes w_{Y_{1,0}^{-1}Y_{n+1-j,1-j}^{-1}Y_{n+2-j,-j}} \text{ for } j < 0.$$

\end{document}